\documentclass[10pt]{article}

\pdfoutput=1
\usepackage[a4paper,top=1.2in,bottom=1in,left=1.2in,right=1.2in]{geometry}

\usepackage{caption}
\usepackage{titlesec}
\renewcommand{\thesubsection}{\thesection.\arabic{subsection}}
\renewcommand{\thesubsubsection}{\thesubsection.\Alph{subsubsection}}
\titleformat{\section}[hang]{\normalfont\bfseries}{\thesection}{.5em}{}
\titlelabel{\subsubsection}{\empty}
\titleformat{\subsection}{\normalfont\bfseries}{\thesubsection.}{.5em}{}[]
\titleformat{\subsubsection}[runin]{\normalfont}{\thesubsubsection.}{0em}{}[]

\renewcommand{\abstract}[1]{{\gdef\thepoabstract{#1}}}

\usepackage{amsthm}
\usepackage{thmtools}

\makeatletter
\renewcommand\maketitle
{
\thispagestyle{empty}
\rmfamily\selectfont
\ifdefined\@title{\noindent\Large\bfseries\centering \@title \par}\else\fi
\vspace{2em}
\ifdefined\@author{\centering\normalfont \@author \par}
\vspace{.5em}
\ifdefined\thepoaffiliation{\noindent\small\thepoaffiliation\par}\fi\vspace{3em}\else\fi
\ifdefined\thepoabstract{\small\noindent{{\bfseries Abstract}.\;\;}\thepoabstract\par\vspace{2em}}\else\fi
\ifdefined\thepokeywords{{\noindent\bfseries Keywords\;\;}\thepokeywords\par\vspace{5em}}\else\fi
\ifdefined\theporuntitle{\fancyhead[L]{\footnotesize\theporuntitle}}\fi
}
\makeatother

\usepackage[usenames,dvipsnames]{color}
\definecolor{RefColor}{rgb}{0,0,.85}
\definecolor{UrlColor}{rgb}{.5,.5,.5}%
\RequirePackage[colorlinks,linkcolor=RefColor,citecolor=RefColor,urlcolor=UrlColor,linktoc=page]{hyperref}
\hypersetup{pdfinfo={Subject={ }}}
\RequirePackage[OT1]{fontenc}
\usepackage{natbib}

\usepackage[english]{babel}

\usepackage{enumitem}
\setlist[itemize]{leftmargin=1.5em}

\usepackage{amsmath,amssymb,amscd,amsfonts,amsthm,mathtools}
\usepackage[noabbrev,capitalize]{cleveref}
\usepackage{dsfont}
\usepackage{thmtools}
\usepackage{nccmath}
\usepackage{scalerel}

\usepackage{bibentry}

\usepackage{booktabs}
\usepackage{tikz}
\usepackage[low-sup]{subdepth}
\usetikzlibrary{calc}
\usetikzlibrary{decorations.markings,decorations.pathreplacing}
\tikzstyle{mybraces}=[mirrorbrace/.style={
          decoration={brace, mirror},
          decorate},brace/.style={
          decoration={brace},
          decorate}]
\usepackage{pst-node}
\usepackage{tikz-cd}

\usepackage{tocloft}

\cftsetpnumwidth{5cm}
\cftsetrmarg{6cm}

\usepackage{wrapfig}
\usepackage{subcaption}
\usepackage[nottoc]{tocbibind}
\settocbibname{References}

\theoremstyle{plain}
\declaretheoremstyle[postheadspace=.4em,headfont=\bfseries,bodyfont=\itshape,spaceabove=8pt,
spacebelow=10pt]{basic}
\theoremstyle{basic}
\declaretheorem[style=basic,name={Theorem}]{theorem}
\declaretheorem[style=basic,sibling=theorem,name={Lemma}]{lemma}

\declaretheorem[style=basic,sibling=theorem,name={Proposition}]{proposition}
\declaretheorem[style=basic,sibling=theorem,name={Corollary}]{corollary}
\theoremstyle{definition}

\declaretheorem[style=definition,name={Remark}]{remark}
\declaretheorem[style=definition,name={Remark},numbered=no]{remark*}

\usepackage{times}

\allowdisplaybreaks

\usetikzlibrary{shapes.geometric}
\usetikzlibrary{arrows.meta}

\makeatletter 
\newcommand{\htarget}[1]{\Hy@raisedlink{\hypertarget{#1}{}}} 
\makeatother

\usepackage{cancel}
\usepackage{nccmath,tikz}
\usepackage{stmaryrd,scalerel} 

\usepackage{float}

\usetikzlibrary{calc}
\usetikzlibrary{decorations.markings,decorations.pathreplacing}
\usetikzlibrary{shapes.misc}

\definecolor{colorhot}{RGB}{203,41,87}
\definecolor{colorcold}{RGB}{66,116,217}
\definecolor{colormoderate}{RGB}{114,62,195}
\definecolor{colornew}{RGB}{255,140,0}

\newcommand{\tikzfiguretextfont}{\sffamily\small}
\newcommand{\tikzfigurefont}{\tikzfiguretextfont}
\tikzset{
    figure text/.style={
        >={Stealth[length=5pt,width=6pt]},
        every node/.append style={font=\tikzfigurefont}
    },
    figure axis/.style={
        ->,
        line width=.35mm
    },
    figure auxiliary/.style={
        line width=.5mm
    }
}

\DeclareDocumentCommand{\hcancel}{mO{0pt}O{0pt}O{0pt}O{0pt}}{%
    \tikz[baseline=(tocancel.base)]{
        \node[inner sep=0pt,outer sep=0pt] (tocancel) {#1};
        \draw[red] ($(tocancel.south west)+(#2,#3)$) -- ($(tocancel.north east)+(#4,#5)$);
    }%
}%

\newenvironment{proplist}{\begin{enumerate}[
    leftmargin=2.5em,
    labelwidth=0em,
    label=(\roman{enumi}),
    topsep=.1em,
    partopsep=0em,
    itemsep=0em
    ]}
{\end{enumerate}}

\DeclareMathOperator*{\argmin}{argmin}

\def\P{{\mathbb P}}

\def\Var{\text{\rm Var}}

\def\N{\mathbb{N}}

\def\P{{\mathbb P}}

\def\R{\mathbb{R}}

\def\cB{\mathcal{B}}

\def\cE{\mathcal{E}}

\def\cN{\mathcal{N}}

\def\cS{\mathcal{S}}

\DeclareMathOperator{\msum}{\medmath\sum}

\DeclareMathOperator{\mint}{\medmath\int}
\newcommand{\argdot}{{\,\vcenter{\hbox{\tiny$\bullet$}}\,}}
\newcommand{\tagaligneq}{\refstepcounter{equation}\tag{\theequation}}

\newcommand{\msup}{\sup\nolimits}

\newcommand{\ind}{\mathbb{I}}

\newcommand{\mean}{\mathbb{E}}

\usepackage{makecell}

\expandafter\def\expandafter\normalsize\expandafter{%
    \normalsize%
    \setlength\abovedisplayskip{4pt}%
    \setlength\belowdisplayskip{4pt}%
    \setlength\abovedisplayshortskip{-8pt}%
    \setlength\belowdisplayshortskip{2pt}%
}

\begin{document}

\title{Top Singular Value in Sum-Products of Random Matrices}

\author{
     Kevin Han Huang{$^*$} \; and \; Boris Hanin{$^\dagger$}
     \\[1em]
     $^*$University of Warwick \quad $^\dagger$Princeton University
     \\[-2em]
}

\begin{abstract}
   {
      We study the top singular value for a sum of $m$ independent $n\times n$ random matrices, each of which is a product of $N$ i.i.d. $n\times n$ Gaussian matrices. Our main conceptual observation is that when  $m,n,N\rightarrow \infty$, the top singular value coincides with the partition function in a random energy model at the inverse temperature  $\beta=\sqrt{2(N-1)/(n\log m)}$, with  energies depending on the ratio $N/n$. We provide several non-asymptotic results making this approximation precise.
   }
\end{abstract}

\maketitle

\vspace{-2em}

\section{Introduction}

This article concerns the top singular value of a sum of products of i.i.d.~Gaussian matrices,
\begin{align*}
    X \;\coloneqq\; \mfrac{1}{\sqrt{m}} \msum_{i \leq m} X_i \;,
    \qquad 
    X_i \;\coloneqq\; X_{iN} \cdots X_{i1}\;,
    \tagaligneq \label{eq:model}
\end{align*}
where $X_{ij}$ are i.i.d. $n\times n$ matrices with i.i.d. $\cN(0, n^{-1})$ entries. We consider the ``triple-scaling'' asymptotic regime 
\begin{align}\label{eq:triple-scaling-limit}
    N, n, m \rightarrow \infty\;,
\end{align}
and study the top Lyapunov exponent $\lambda_1(X)$ (or equivalently the top singular value $s_1(X)$)
\begin{align}
  \lambda_1(X)=\mfrac{1}{N}  \log s_1(X) :=   \sup_{\theta \in \cS^{n-1}}  \mfrac{1}{N} \log \| X \theta \|,
    \tagaligneq \label{eq:top:Lyapunov}
\end{align}
where $\cS^{n-1} \subset \R^n$ denotes the unit sphere. For a single matrix product ($m=1$), \cite{hanin2021non} showed that the supremum above can be removed, in the sense that for any fixed $\theta$, with high probability,
\begin{align*}
     \sup_{\theta' \in \cS^{n-1}} \mfrac{1}{N} \log \| X_1 \theta' \|
 \approx
    \mfrac{1}{N} \log \| X_1 \theta \|
     \;
\end{align*}
as soon as $N\gg \log n$. We will obtain an analogous statement (see \Cref{lem:sup:to:pointwise}) that this holds for our model when $N,n,m\rightarrow \infty$ under the same condition. As such, to study the top singular value, we first seek to study the random variable
\begin{align*}
    \log \| X \theta \|  \qquad \text{ for any fixed } \theta \in \cS^{n-1}\;,
\end{align*}
which characterizes the effect of $X$ on a fixed $1$-dimensional subspace. 

\vspace{.5em}

A key observation is that when $n,N,m\gg 1$, the random variable $\log \| X \theta \|$ is well-approximated by 
\[
Z_{m,n,N} \;=\; - \mfrac{N-1}{2n} + \mfrac{1}{2} \log \mfrac{1}{m} \msum_{i=1}^m e^{-\beta \sqrt{\log m} E_i}
\;,
\]
which is, up to a recentering and rescaling, the log-partition function of a random energy model (REM) at inverse temperature 
\begin{align*}
    \beta 
    \;\coloneqq\;
    \mfrac{\sqrt{2 (N-1)}}{\sqrt{n \log m}}
    \;.
    \tagaligneq \label{eq:expression:beta}
\end{align*}
The REM involves $m$ i.i.d.~non-Gaussian energies $(E_i)_{i =1,\ldots,m}$, whose distribution depends on $N/n$, each corresponding to the contribution of one $N$-fold random matrix product $X_i$. The precise statement is in \Cref{cor:log:Hanson:Wright} and \Cref{eq:empirical:REM:approx}. In the triple scaling limit \eqref{eq:triple-scaling-limit}, the log-partition function $Z_{m,n,N}$ coincides with the limiting log-partition function of a Gaussian random energy model, given (after the same renormalization) by
\begin{align*}
    Z
    \;\coloneqq&\;
    \begin{cases}
            0
            &
            \text{ if }
            \beta \leq \sqrt{2}
            \;,
            \\
            - \frac{(\beta - \sqrt{2})^2}{4} \, \log m
            & 
            \text{ if } 
            \beta > \sqrt{2}
            \;.
    \end{cases}
   \tagaligneq \label{eq:limit}
\end{align*}
This reveals a surprisingly complex phase diagram (Figure \ref{fig:double:scaling}) for the top singular value of $X$. Notably, $s_1(X)$ undergoes a phase transition as the inverse temperature parameter $\beta$ crosses the threshold $\sqrt{2}$: In the high-temperature regime when $\beta \leq \sqrt{2}$, $s_1(X)$ is characterized by the massive number of configurations with typical energies, i.e.~many matrix products contribute to the value of $\log \| X \theta \|$. In the low-temperature regime when $\beta > \sqrt{2}$, $s_1(X)$ is dominated by a small number of configurations with excessively low energies, i.e.~a few matrix products $X_i$'s with small values of $\|X_i \theta\|$ dominate. The next result makes this formal.

\begin{theorem}[Effect of $X$ on a fixed $1$-dimensional subspace] \label{thm:pointwise} There exist universal constants $c, C > 0$ such that the following holds for any $\epsilon \in (0,1)$:
\begin{proplist}
    \item If $\beta = o(1)$, then with probability at least $1 - \epsilon^{-2} e^{ - c \log m } - 2 e^{-cn\epsilon^2}$, we have
    \begin{align*}
        \big|  
            \log \| X \theta \|
            -
            Z
        \big| 
        \;\leq\; 
        C
        \Big(
            \mfrac{\epsilon}{1-\epsilon}    
            +
            \mfrac{\beta^2 \log m}{n}
        \Big)
        \;;
    \end{align*}
    \item If $\beta = \Omega(1)$, $\log m = o(N^{1/3})$ and $N=o(n^3)$, then  
    \begin{align*}
        \mfrac{1}{\beta^2 \log m}
        \big|  
            \log \| X \theta \|
            -
            Z
        \big| 
        \;\leq\; 
        C \Big( \mfrac{\epsilon}{(1-\epsilon) \beta \log m}
                    +
                    \mfrac{1}{ \beta (\log m)^{1/4}}
                \Big)
        \;
    \end{align*}
    with probability at least $1 - (1+ \epsilon^{-2}) e^{ - c  \, (\log m)^{3/4} } - 2e^{-cn\epsilon^2}$.
\end{proplist}
\end{theorem}

\begin{remark}[Conditions arising from the Gaussian approximation of non-Gaussian energies] The condition $\log m = o(N^{1/3})$ arises because we approximate the individual non-Gaussian energies by Gaussians using Cr\'amer's moderate deviation theorem \citep{cramer1938nouveau} at a location $x \sim \sqrt{\log m}$, and Cr\'amer's theorem introduces a condition $x = o(N^{1/6})$. The condition $N=o(n^3)$ arises from a first-order approximation of the mean and variance of the non-Gaussian energies. We conjecture that both conditions are improvable by a finer approximation of the energies.
\end{remark}

\vspace{.5em}

By approximating the top Lyapunov exponent of $X$ by $\log \| X \theta \|$ for some fixed $\theta \in \cS^{n-1}$, we obtain the following:

\begin{theorem}[Top Lyapunov exponent of $X$] \label{thm:top:Lyapunov}  There exist universal constants $c, C, C' > 0$ such that the following holds for any $\epsilon \in (0,1)$ and $\alpha > 0$:
\begin{proplist}
    \item If $\beta = o(1)$, then with probability $1 - \epsilon^{-2} e^{ - c \log m } - 2 e^{-cn\epsilon^2} - C' n^{-\alpha}$, we have
    \begin{align*}
        \big|  
            \log s_1(X)
            -
            Z
        \big| 
        \;\leq\; 
        C
        \Big(
            \mfrac{\epsilon}{1-\epsilon}
            +
            \mfrac{\beta^2 \log m}{n}
        \Big)
        +
        \mfrac{(1+2\alpha) \log n}{2} 
        \;;
    \end{align*}
    \item If $\beta = \Omega(1)$,  $\log m = o(N^{1/3})$ and $N=o(n^3)$, then  with probability $1 - (1+ \epsilon^{-2}) e^{ - c  \, (\log m)^{3/4} } - 2e^{-cn\epsilon^2} - C' n^{-\alpha / 2}$, we have 
    \begin{align*}
        \mfrac{1}{\beta^2 \log m}
        \big|  
            \log s_1(X)
            -
            Z
        \big| 
        \;\leq\; 
        C \Big( \mfrac{\epsilon}{(1-\epsilon) \beta \log m}
                    +
                    \mfrac{1}{\beta (\log m)^{1/4}}
                \Big)
        +
        \mfrac{(1+2\alpha) \log n}{2 \beta^2 \log m}
        \;.
    \end{align*}
\end{proplist}
\end{theorem}

\vspace{.5em}

\begin{figure}[t]
    
    \def\phasefigwidth{15}
    \def\phasefigheight{15}
    \centering
    \begin{tikzpicture}[scale=.36, figure text]
        \draw[figure axis] (0,0) -- ({\phasefigwidth+1},0) node[right] {$N$};
        \draw[figure axis] (0,0) -- (0,{\phasefigheight+1}) node[above] {$n$};
        
        \draw[line width=.2em,color=colorhot] (0,\phasefigheight) -- ({.5*\phasefigwidth},\phasefigheight);
        \node[align=left, anchor=north west] at (0,{\phasefigheight}) {
                \scriptsize \textcolor{colorhot}{Free Probability}
                \\
                \scriptsize \textcolor{colorhot}{\emph{(fixed $N$)}}
        };

        \node[align=left, anchor=south west, draw=black, inner sep=2pt] at  (0.9,{\phasefigheight+.8}) {
                \scriptsize \textcolor{colorhot}{$\frac{\log (N+1)}{2}
                + \frac{N}{2} \log \big( 1 + \frac{1}{N} \big) + o_\P(1)$,}
                \\
                \scriptsize {Different from $\log \| X \theta\| \overset{d}{\approx} \cN(0, \frac{N}{2n} )$}
        };

        \draw[line width=.2em,color=colorcold] (\phasefigwidth,{.5*\phasefigheight}) -- (\phasefigwidth,0);
        \node[align=right, anchor=east] at ({\phasefigwidth},{.2*\phasefigheight}) {
                \scriptsize \textcolor{colorcold}{Ergodic}
                \\
                \scriptsize \textcolor{colorcold}{Theory}
                \\
                \scriptsize \textcolor{colorcold}{\emph{(fixed $n$)}}
        };

        \node[align=left, anchor=west, draw=black, inner sep=2pt] at ({\phasefigwidth+.8},{.2*\phasefigheight}) {
                \scriptsize \textcolor{colorcold}{$N \big(\frac{\log (2/n)}{2} + \frac{\psi(n/2)}{2}  + o_\P(1) \big)$,}
                \\
                \scriptsize {Same as $\log \|X \theta \|$}
        };

        \draw[line width=.2em, color=colormoderate] ({.5*\phasefigwidth},\phasefigheight) .. controls ({.62*\phasefigwidth},\phasefigheight) and ({.72*\phasefigwidth},{.88*\phasefigheight}) .. ({.8*\phasefigwidth},{.8*\phasefigheight}) .. controls ({.86*\phasefigwidth},{.74*\phasefigheight}) and (\phasefigwidth,{.62*\phasefigheight}) .. (\phasefigwidth,{.5*\phasefigheight});

        \draw[->, dotted, figure auxiliary, color=colormoderate] (0,0) -- ({.8*\phasefigwidth},{.8*\phasefigheight});
        
        \node[anchor=east] at ({.5*\phasefigwidth},{.55*\phasefigheight}) {\textcolor{colormoderate}{$\frac{N}{n} \rightarrow \gamma \in (0, \infty)$}};

        \node[align=left, anchor=south west, draw=black, inner sep=4pt] at ({.8*\phasefigwidth+1.5},{.8*\phasefigheight-.2}) {
                \scriptsize \textcolor{colormoderate}{
                 Unknown
                }
        };

    \end{tikzpicture} 
    \caption{
        Known results on the value of $ \log s_1(X)$ when $m=1$ under the single scaling regime, where only one of $n$ and $N$ grows with the other fixed. $\psi$ denotes the digamma function. For the double-scaling regime, where $N/n \rightarrow \gamma \in (0,\infty)$, results are only known for (i) the complex Gaussian case (Theorem 1.2 of \cite{liu2023lyapunov}): $\log s_1(X) \approx \frac{\log n}{2} + F(\gamma) + o_\P(1 )$, where $F(\gamma)$ is a generic quantity that depends only on $\gamma$; (ii) for a fixed  $\theta \in \cS^{n-1}$, where $\log \| X \theta\| \overset{d}{\rightarrow} \cN( - \frac{\gamma}{2}, \frac{\gamma}{2}  ) $.
    }
    \label{fig:double:scaling}
\end{figure}
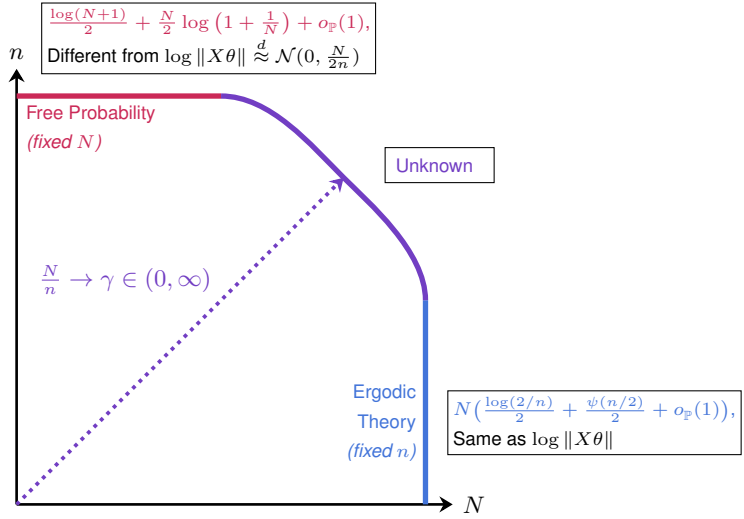

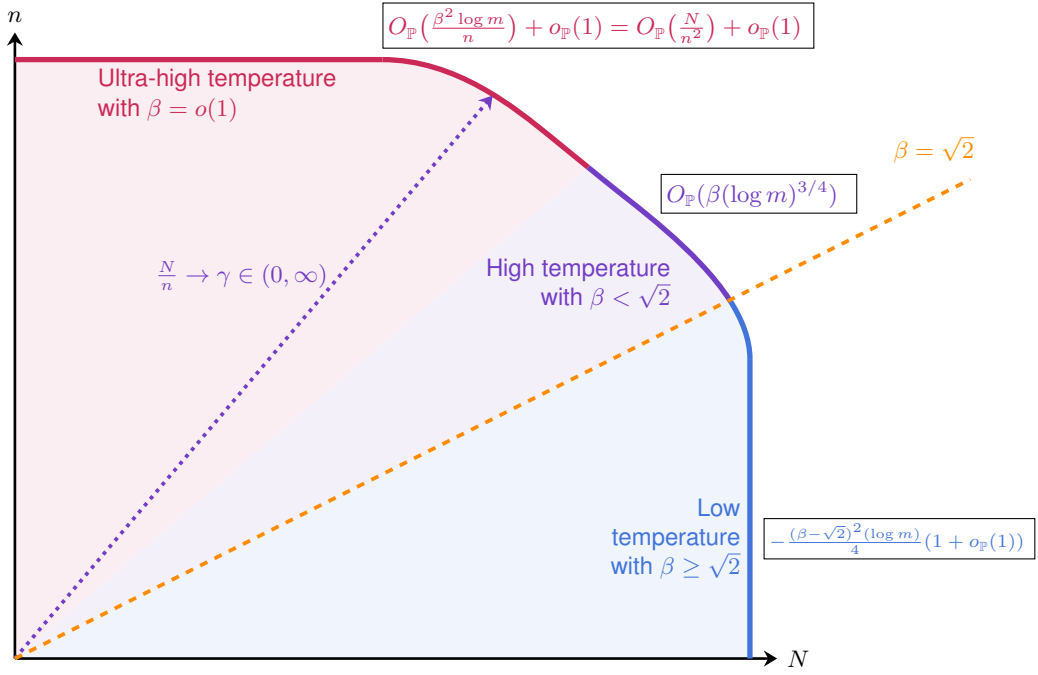
\begin{figure}[t]
    
    \def\phasefigwidth{27}
    \def\phasefigheight{22}
    \centering
    \begin{tikzpicture}[scale=.36, figure text]
        \coordinate (phaseorangestart) at ({.73*\phasefigwidth},{.87*\phasefigheight});
        \coordinate (phaseorangeupper) at ({.78*\phasefigwidth},{.82*\phasefigheight});
        \coordinate (phasebottomdivider) at ({.97245532*\phasefigwidth},{.59843404*\phasefigheight});
        \begin{scope}
            \clip (0,0) -- (0,\phasefigheight) -- ({.5*\phasefigwidth},\phasefigheight)
                .. controls ({.6*\phasefigwidth},\phasefigheight) and ({.68*\phasefigwidth},{.92*\phasefigheight}) .. (phaseorangestart)
                .. controls ({.75*\phasefigwidth},{.85*\phasefigheight}) and ({.77*\phasefigwidth},{.83*\phasefigheight}) .. (phaseorangeupper)
                .. controls ({.79*\phasefigwidth},{.81*\phasefigheight}) and ({.795*\phasefigwidth},{.805*\phasefigheight}) .. ({.8*\phasefigwidth},{.8*\phasefigheight})
                .. controls ({.84338186*\phasefigwidth},{.75661814*\phasefigheight}) and ({.92858562*\phasefigwidth},{.68186986*\phasefigheight}) .. (phasebottomdivider)
                .. controls ({.98926034*\phasefigwidth},{.56647257*\phasefigheight}) and (\phasefigwidth,{.53323628*\phasefigheight}) .. (\phasefigwidth,{.5*\phasefigheight})
                -- (\phasefigwidth,0) -- cycle;
            \fill[colorcold!8] (0,0) -- ({1.3*\phasefigwidth},{.8*\phasefigheight}) -- (\phasefigwidth,0) -- cycle;
            \fill[colormoderate!8] (0,0) -- ({1.3*\phasefigwidth},{.8*\phasefigheight}) -- (phaseorangeupper) -- cycle;
            \fill[colorhot!8] (0,0) -- (phaseorangeupper)
                -- ({\phasefigwidth+2},{\phasefigheight+2}) -- (0,{\phasefigheight+2}) -- cycle;
        \end{scope}

        \draw[figure axis] (0,0) -- ({\phasefigwidth+1},0) node[right] {$N$};
        \draw[figure axis] (0,0) -- (0,{\phasefigheight+1}) node[above] {$n$};
        
        \draw[line width=.2em,color=colorhot] (0,\phasefigheight) -- ({.5*\phasefigwidth},\phasefigheight);
        \node[align=left, anchor=north west] at ({\phasefigwidth*0.1},{\phasefigheight}) {
                 \textcolor{colorhot}{Ultra-high temperature}
                 \\
                 \textcolor{colorhot}{with $\beta=o(1)$}
        };

        \node[align=right, anchor=north east, inner sep=2pt] at  ({\phasefigwidth*0.9},{\phasefigheight*0.7-.5}) {
               \textcolor{colormoderate}{High temperature}
               \\
               \textcolor{colormoderate}{with $\beta < \sqrt{2}$}
        };

        \node[align=left, anchor=south west, draw=black, inner sep=2pt] at  ({\phasefigwidth*0.5},{\phasefigheight+.4}) {
                 \textcolor{colorhot}{$
                 O_\P\big( \frac{\beta^2 \log m}{n} \big)
                 +
                 o_\P(1) 
                 =
                  O_\P\big( \frac{N}{n^2} \big)
                  +
                  o_\P(1) 
                 $}
        };

        \draw[line width=.2em,color=colorcold] (\phasefigwidth,{.5*\phasefigheight}) -- (\phasefigwidth,0);
        \node[align=right, anchor=east] at ({\phasefigwidth},{.2*\phasefigheight}) {
                 \textcolor{colorcold}{Low}
                \\
                 \textcolor{colorcold}{temperature}
                 \\
                 \textcolor{colorcold}{with $\beta \geq \sqrt{2}$}
        };

        \node[align=left, anchor=west, draw=black, inner sep=2pt] at ({\phasefigwidth+.5},{.2*\phasefigheight}) {
                \scriptsize \textcolor{colorcold}{$
                    - \frac{(\beta - \sqrt{2})^2 (\log m )}{4} ( 1 + o_\P(1))
                    $}
        };

        \draw[line width=.2em, color=colorhot] ({.5*\phasefigwidth},\phasefigheight)
            .. controls ({.6*\phasefigwidth},\phasefigheight) and ({.68*\phasefigwidth},{.92*\phasefigheight}) .. (phaseorangestart)
            .. controls ({.75*\phasefigwidth},{.85*\phasefigheight}) and ({.77*\phasefigwidth},{.83*\phasefigheight}) .. (phaseorangeupper);

        \draw[->, dotted, figure auxiliary, color=colormoderate] (0,0) -- ({.65*\phasefigwidth},{.94*\phasefigheight});

        \node[anchor=south east, align=left] at ({.47*\phasefigwidth-.6},{.6*\phasefigheight}) {
            \textcolor{colormoderate}{$\frac{N}{n} \rightarrow \gamma \in (0, \infty)$}
        };

        \draw[line width=.2em, color=colormoderate]  (phaseorangeupper)
            .. controls ({.79*\phasefigwidth},{.81*\phasefigheight}) and ({.795*\phasefigwidth},{.805*\phasefigheight}) .. ({.8*\phasefigwidth},{.8*\phasefigheight})
            .. controls ({.84338186*\phasefigwidth},{.75661814*\phasefigheight}) and ({.92858562*\phasefigwidth},{.68186986*\phasefigheight}) .. (phasebottomdivider);

        \draw[line width=.2em, color=colorcold]  (phasebottomdivider)
            .. controls ({.98926034*\phasefigwidth},{.56647257*\phasefigheight}) and (\phasefigwidth,{.53323628*\phasefigheight}) .. (\phasefigwidth,{.5*\phasefigheight});

        \draw[dashed, figure auxiliary, color=colornew] (0,0) -- ({1.3*\phasefigwidth},{0.8*\phasefigheight});

        \node[inner sep=2pt] at ({1.25*\phasefigwidth},{0.85*\phasefigheight}){\textcolor{colornew}{$\beta=\sqrt{2}$}};

        \node[align=left, anchor=south west, draw=black, inner sep=2pt] at ({.86*\phasefigwidth+.5},{.72*\phasefigheight+.5}) {
                 \textcolor{colormoderate}{$O_\P( \beta(\log m)^{3/4} ) $ 
                }
        };

    \end{tikzpicture}
    \caption{
        Our results about the value of $\log \|X \theta \|$ under the triple scaling where $m$ grows with $n$ and $N$, which is determined not by $\lim \frac{N}{n} = \gamma$ but by the inverse temperature $\beta = \frac{\sqrt{2(N-1)}}{\sqrt{n \log m}} $.
    }
\end{figure}

\begin{figure}[t]
    
    \def\phasefigwidth{27}
    \def\phasefigheight{22}
    \centering
    \begin{tikzpicture}[scale=.36, figure text]
        \coordinate (phaseorangestart) at ({.73*\phasefigwidth},{.87*\phasefigheight});
        \coordinate (phaseorangeupper) at ({.78*\phasefigwidth},{.82*\phasefigheight});
        \coordinate (phasebottomdivider) at ({.97245532*\phasefigwidth},{.59843404*\phasefigheight});
        \begin{scope}
            \clip (0,0) -- (0,\phasefigheight) -- ({.5*\phasefigwidth},\phasefigheight)
                .. controls ({.6*\phasefigwidth},\phasefigheight) and ({.68*\phasefigwidth},{.92*\phasefigheight}) .. (phaseorangestart)
                .. controls ({.75*\phasefigwidth},{.85*\phasefigheight}) and ({.77*\phasefigwidth},{.83*\phasefigheight}) .. (phaseorangeupper)
                .. controls ({.79*\phasefigwidth},{.81*\phasefigheight}) and ({.795*\phasefigwidth},{.805*\phasefigheight}) .. ({.8*\phasefigwidth},{.8*\phasefigheight})
                .. controls ({.84338186*\phasefigwidth},{.75661814*\phasefigheight}) and ({.92858562*\phasefigwidth},{.68186986*\phasefigheight}) .. (phasebottomdivider)
                .. controls ({.98926034*\phasefigwidth},{.56647257*\phasefigheight}) and (\phasefigwidth,{.53323628*\phasefigheight}) .. (\phasefigwidth,{.5*\phasefigheight})
                -- (\phasefigwidth,0) -- cycle;
            \fill[colorcold!8] (0,0) -- ({1.3*\phasefigwidth},{.8*\phasefigheight}) -- (\phasefigwidth,0) -- cycle;
            \fill[colormoderate!8] (0,0) -- ({1.3*\phasefigwidth},{.8*\phasefigheight}) -- (phaseorangeupper) -- cycle;
            \fill[colorhot!8] (0,0) -- (phaseorangeupper)
                -- ({\phasefigwidth+2},{\phasefigheight+2}) -- (0,{\phasefigheight+2}) -- cycle;
        \end{scope}

        \draw[figure axis] (0,0) -- ({\phasefigwidth+1},0) node[right] {$N$};
        \draw[figure axis] (0,0) -- (0,{\phasefigheight+1}) node[above] {$n$};
        
        \draw[line width=.2em,color=colorhot] (0,\phasefigheight) -- ({.5*\phasefigwidth},\phasefigheight);
        \node[align=left, anchor=north west] at ({\phasefigwidth*0.1},{\phasefigheight}) {
                 \textcolor{colorhot}{Ultra-high temperature}
                 \\
                 \textcolor{colorhot}{with $\beta=o(1)$}
        };

        \node[align=right, anchor=north east, inner sep=2pt] at  ({\phasefigwidth*0.9},{\phasefigheight*0.7-.5}) {
               \textcolor{colormoderate}{High temperature}
               \\
               \textcolor{colormoderate}{with $\beta < \sqrt{2}$}
        };

        \node[align=left, anchor=south west, draw=black, inner sep=2pt] at  ({\phasefigwidth*0.5},{\phasefigheight+.4}) {
                 \textcolor{colorhot}{$O_\P\big( \log n + \frac{\beta^2 \log m}{n} \big)$}
        };

        \draw[line width=.2em,color=colorcold] (\phasefigwidth,{.5*\phasefigheight}) -- (\phasefigwidth,0);
        \node[align=right, anchor=east] at ({\phasefigwidth},{.2*\phasefigheight}) {
                 \textcolor{colorcold}{Low}
                \\
                 \textcolor{colorcold}{temperature}
                 \\
                 \textcolor{colorcold}{with $\beta \geq \sqrt{2}$}
        };

        \node[align=left, anchor=west, draw=black, inner sep=2pt] at ({\phasefigwidth+.5},{.2*\phasefigheight}) {
                \scriptsize \textcolor{colorcold}{$
                    - \frac{(\beta - \sqrt{2})^2(\log m )}{4} $}
                \\
                \scriptsize \quad \textcolor{colorcold}{$ \times \Big( 1 + o_\P(1) + 
                    O_\P\Big(
                    \mfrac{\log n}{\beta^2 \log m} \Big)
                    \Big)
                    $}
        };

        \draw[line width=.2em, color=colorhot] ({.5*\phasefigwidth},\phasefigheight)
            .. controls ({.6*\phasefigwidth},\phasefigheight) and ({.68*\phasefigwidth},{.92*\phasefigheight}) .. (phaseorangestart)
            .. controls ({.75*\phasefigwidth},{.85*\phasefigheight}) and ({.77*\phasefigwidth},{.83*\phasefigheight}) .. (phaseorangeupper);

        \draw[->, dotted, figure auxiliary, color=colormoderate] (0,0) -- ({.65*\phasefigwidth},{.94*\phasefigheight});

        \node[anchor=south east, align=left] at ({.47*\phasefigwidth-.6},{.6*\phasefigheight}) {
            \textcolor{colormoderate}{$\frac{N}{n} \rightarrow \gamma \in (0, \infty)$}
        };

        \draw[line width=.2em, color=colormoderate]  (phaseorangeupper)
            .. controls ({.79*\phasefigwidth},{.81*\phasefigheight}) and ({.795*\phasefigwidth},{.805*\phasefigheight}) .. ({.8*\phasefigwidth},{.8*\phasefigheight})
            .. controls ({.84338186*\phasefigwidth},{.75661814*\phasefigheight}) and ({.92858562*\phasefigwidth},{.68186986*\phasefigheight}) .. (phasebottomdivider);

        \draw[line width=.2em, color=colorcold]  (phasebottomdivider)
            .. controls ({.98926034*\phasefigwidth},{.56647257*\phasefigheight}) and (\phasefigwidth,{.53323628*\phasefigheight}) .. (\phasefigwidth,{.5*\phasefigheight});

        \draw[dashed, figure auxiliary, color=colornew] (0,0) -- ({1.3*\phasefigwidth},{0.8*\phasefigheight});

        \node[inner sep=2pt] at ({1.32*\phasefigwidth},{0.85*\phasefigheight}){\textcolor{colornew}{$\beta=\sqrt{2}$}};

        \node[align=left, anchor=south west, draw=black, inner sep=2pt] at ({.82*\phasefigwidth+.5},{.75*\phasefigheight+.5}) {
                 \textcolor{colormoderate}{$O_\P( \log n + \beta(\log m)^{3/4} ) $ 
                }
        };

    \end{tikzpicture}
    \caption{
        Our results about the value of $\log s_1(X)$ under the triple scaling where $m$ grows with $n$ and $N$, which is determined not by $\lim \frac{N}{n} = \gamma$ but by the inverse temperature $\beta = \frac{\sqrt{2(N-1)}}{\sqrt{n \log m}} $.
    }
\end{figure}
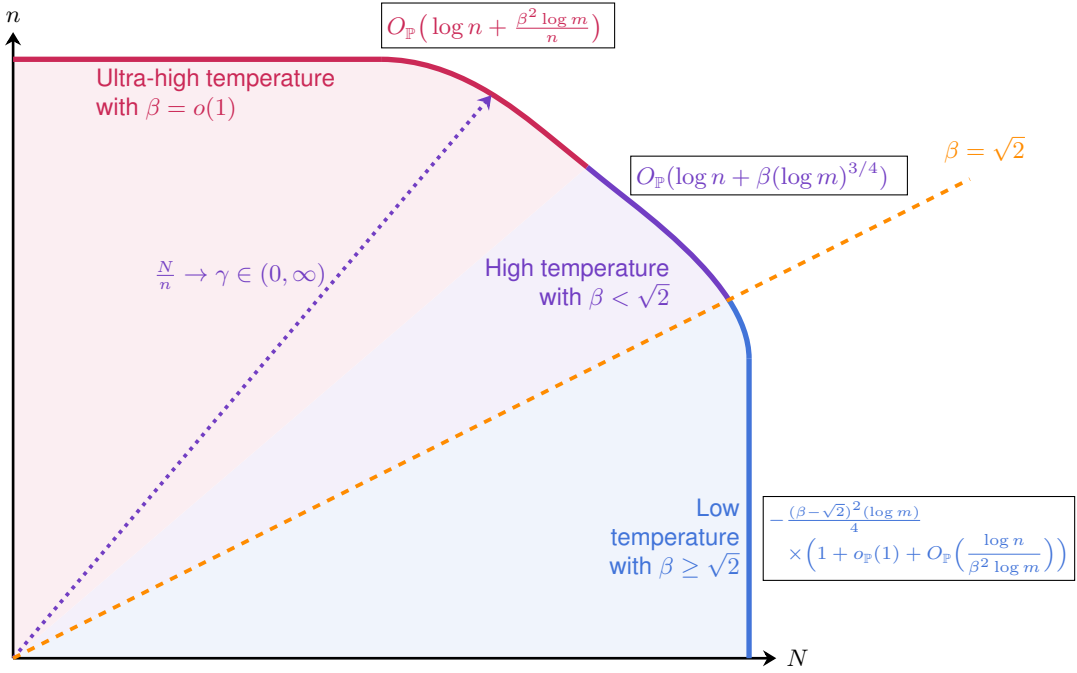

We now interpret the results in both the high and low temperature regimes:

\vspace{.5em}

\noindent
\textbf{High temperature regime with $\beta \leq \sqrt{2}$.} In this case, the limiting approximation is $Z =  0$, and \Cref{thm:pointwise,thm:top:Lyapunov} imply the following about $\log \| X \theta \|$ and $\log s_1(X)$:
\begin{enumerate}
    \item[(a)] If $\beta = o(1)$, by choosing $\epsilon = m^{ - c_1}+ n^{- \frac{1}{2} + c_2}$ for some small $c_1, c_2 > 0$, we obtain that with high probability,
    \begin{align*}
        \log \| X \theta \| 
        \;=&\;
        O\Big( m^{-c_1} + n^{ -\frac{1}{2} + c_2 }  +  \mfrac{\beta^2 \log m}{n}  \Big)
        \;.
    \end{align*}
    Therefore with high probability, 
    \begin{align*}
        \| X \theta \| \;=\; 1 + o(1)
        \;\qquad\;
        \text{ provided that }
         \beta^2 \log m = o(n)\;.
    \end{align*}
    Since $\beta^2 = \Theta(\frac{N}{n \log m} ) = o(1)$, the condition above can be satisfied by either 
    \begin{align*}
        \log m \;=&\; O(n)
        &\text{ or }&&
        N \;=&\; o (n^2)\;.
    \end{align*}
    Meanwhile, by the same argument, we have that with high probability,
    \begin{align*}
        \log s_1(X)
        \;=&\;
        O\Big( 
            m^{-c_1} + n^{ -\frac{1}{2} + c_2 }  +  \mfrac{\beta^2 \log m}{n}  
            +
            \log n
        \Big)
        \;=\;
        O\Big(\log n +  \mfrac{\beta^2 \log m}{n}   \Big)
        \;.
    \end{align*}
    Therefore provided that $\beta^2 \log m = o(n \log n)$,  we have that with high probability, 
    \begin{align*}
        s_1(X) \;\in\; \big[ n^{- c' } \,,\, n^{c'}  \big]
    \end{align*}
    for some universal constant $c' > 0$.
    \item[(b)] If $\beta = \Omega(1)$ with $\beta < \sqrt{2}$, by the same choice of $\epsilon$ and noting that $N/(n \log m) = \Theta(\beta^2)$, we obtain that with high probability, 
    \begin{align*}
        \log \| X \theta \| 
        \;=&\;
        O\Big( 
            \beta \big( m^{-c_1}+ n^{ -\frac{1}{2} + c_2 } \big) 
            + 
            \beta (\log m)^{3/4} 
        \Big)
        \;=\; O\big( \beta (\log m)^{3/4}  \big)
        \;.
    \end{align*}
    Therefore with high probability, 
    \begin{align*}
        \| X \theta \| \;\in\; \big[ e^{ - c'' \beta (\log m)^{3/4}} \,,\,    e^{ c'' \beta (\log m)^{3/4}} \big]
    \end{align*}
    for some universal constant $c'' > 0$. Similarly, with high probability, 
    \begin{align*}
        \log s_1(X) \;=\; O\big( \beta (\log m)^{3/4} + \log n \big)\;,
    \end{align*}
    in which case 
     \begin{align*}
        s_1(X) \;\in\; \big[ n^{-c'} e^{ - c'' \beta (\log m)^{3/4}} \,,\,   n^{c'} e^{ c'' \beta (\log m)^{3/4}} \big]
        \;.
    \end{align*}
\end{enumerate}
In summary, in the high temperature regime, our results provide an interval that the top singular value $s_1(X)$ lives in with high probability, while providing a precise characterization of $\| X \theta \|$ when we additionally have $\beta = o(1)$.

\vspace{.5em}

\noindent
\textbf{Low temperature regime with $\beta > \sqrt{2}$ and $| \beta - \sqrt{2}| = \Omega(1)$.} The limiting approximation satisfies
\begin{align*}
    Z \;=\; \Theta( \beta^2 \log m )\;.
\end{align*}    
In this case, our result does precisely characterize $\| X \theta \|$: By noting that $\beta^2 \log m = \Theta(\frac{N}{n})$ and redefining the universal constant $C  >0$, the bound in \Cref{thm:pointwise}(ii) reads
\begin{align*}
     \mfrac{1}{\beta^2 \log m}
        \big|  
            \log \| X \theta \|
            -
            Z
        \big| 
        \;\leq\; 
        C \Big( 
            \mfrac{\epsilon}{(1-\epsilon) \beta \log m}
            +
            \mfrac{1}{\beta (\log m)^{1/4}}
        \Big)
        \;.
\end{align*}
Choosing $\epsilon = \frac{1}{2}$, we obtain that with high probability,
\begin{align*}
    \mfrac{| \log \| X \theta \| - Z |}{\beta^2 \log m}
    \;=&\;
    O\Big(  \mfrac{1}{\beta (\log m)^{1/4}}  \Big)
    \;,
\end{align*}
in which case
\begin{align*}
    \log \| X \theta \| \;=&\; Z (1 + o(1))
    &\text{ and }&&
    \| X \theta \| \;=&\; e^{Z (1 + o(1)) }\;.
\end{align*}
By a similar argument, we have that with high probability,
\begin{align*}
    \mfrac{| \log s_1(X) - Z |}{\beta^2 \log m}
    \;=\;
    O\Big(
        \mfrac{1}{\beta (\log m)^{1/4}} 
        +
        \mfrac{\log n}{\beta^2 \log m}
    \Big)
    \;.
\end{align*}
This implies that with high probability,
\begin{align*}
    \log s_1(X) \;=&\; Z (1 + o(1))
    &\text{ and }&&
    s_1(X) \;=&\; e^{Z (1 + o(1)) }
\end{align*}
provided that $\log n = o(\beta^2 \log m)$, which can be satisfied by either 
\begin{align*}
    n \;=&\; o(m)
    &\text{ or }&&
    n \log n \;=&\; o(N)\;.
\end{align*}

\paragraph{Related works.} A large body of work has studied the model \eqref{eq:model} in the case $m=1$ and derived approximations for both $s_1(X)$ and $\| X \theta \|$ with $\theta \in \cS^{n-1}$:
\begin{itemize}
    \item For $N$ fixed and $n \rightarrow \infty$, the top singular value of $X$ is studied in the random matrix theory and free probability literature: The single matrix case ($N=1$) is addressed by works dating back to \cite{geman1980limit} and \cite{yin1988limit}, whereas the fixed-$N$ product of real Gaussian matrices is considered in \cite{akemann2013products} and \cite{saada2024simple}. 
    \item For $n$ fixed and $N \rightarrow \infty$, many classical results are available in the ergodic theory literature \citep{furstenberg1960products,oseledets1968multiplicative}, which notably shows that  $\frac{1}{N} \log s_1(X)$ and $\frac{1}{N} \log \| X \theta \|$ have the same asymptotic limit. The sequential limit of taking $N \rightarrow \infty$ first before taking $n \rightarrow \infty$ is also studied in a long line of works \citep{cohen1984stability,newman1986distribution,isopi1992triangle,kargin2014largest}. 
    \item It has been noted that the $N \rightarrow \infty$ and $n \rightarrow \infty$ limits do not commute at the local scale, at least in the complex Gaussian case \citep{akemann2014universal,liu2023lyapunov}. For the double-scaling regime where $n, N \rightarrow \infty$ with $N/n \rightarrow \gamma \in(0, \infty)$, to the best of our knowledge, precise characterization of the top Lyapunov exponent of products of large \emph{real} Gaussian matrices remains unknown. Nevertheless, related works have characterized the complex Gaussian case \citep{liu2023lyapunov,akemann2019integrable,akemann2020universality}, the case with truncated unitary matrices and some more general complex matrices \citep{Ahn2022,ahn2022extremal} and the fixed direction case, i.e.~$\frac{1}{N} \log \| X \theta \|$ for a fixed $\theta \in \cS^{n-1}$ and $m=1$ (see e.g.~\citep{hanin2020products}, where the argument follows from a direct distributional characterization through log Gamma random variables). In all these cases, the limiting expressions are characterized completely by the limiting ratio $\gamma$. For the real Gaussian case, global laws that are independent of $\gamma$ have also been established in \citep{hanin2021non,hanin2025global}. \Cref{fig:double:scaling} provides an overview of the results in the different regimes.
\end{itemize}

\vspace{.5em}

Much less is known about the model \eqref{eq:model} in the case of a general $m$. A long line of work has studied general non-commutative polynomials of large random matrices \citep{haagerup2005new,schultz2005non,van2025strong}  and more general matrix-valued functions of random matrices \citep{gotze2015asymptotic} --- of which our model \eqref{eq:model} is a special case --- but only in the limit $n \rightarrow \infty$. \cite{bordenave2011spectrum,kosters2015limiting} consider models with sums and products of random matrices, but also in the $n \rightarrow \infty$ limit and with a focus on the global law. Notably, \cite{kosters2015limiting} observe that the $m$-fold average of matrix products has the same global law as a single matrix product. Our results consider a different regime where $N, n, m \rightarrow \infty$ simultaneously, and observe that when $\log m$ is sufficiently large, the top Lyapunov exponent of $X$ behaves very differently from that of a single matrix product.

\vspace{.5em}

We also include a detailed comparison to results from the random energy model literature in \Cref{sec:literature:REM}.

\paragraph{Proof techniques. } Our proof consists of three ingredients. The first ingredient is the reduction of the quantities  $\log s_1(X)$ and $\log \| X \theta\|$ to the log-partition function of a random energy model (REM) with non-Gaussian energies. This is achieved by interlacing a concentration inequality over $m$ random matrices with results from \cite{hanin2021non} that characterize behaviors in the double-scaling regime (i.e. $m=1$). The second ingredient is an approximation of the non-Gaussian REM by a Gaussian REM. This requires sharp location-dependent Gaussian approximations, which are achieved by the classical Cram\'er-type moderate deviation theorem \citep{cramer1938nouveau}. The final ingredient is a set of concentration inequalities for the log-partition function of the REM with quantitative estimates in different regimes. Most of the technical work goes to the moderate-to-low temperature regime with $\beta = \Omega(1)$, where we employ Laplace's method with an explicit computation of the approximation errors. For the ultra-high temperature regime with $\beta =o(1)$, our result is obtained by combining moment generating function estimates of the energies with the Markov inequality.

\paragraph{Organization.} The rest of the article is organized as follows. \Cref{sec:reduction:REM} reduces the problem of characterizing $\log s_1(X)$ and $\log \| X \theta\|$ for a fixed $\theta \in \cS^{n-1}$ to the study of a non-Gaussian random energy model (REM), which will be shown to be approximable by a Gaussian REM. \Cref{sec:REM} presents two concentration inequalities on the log-partition function of the non-Gaussian REM (\Cref{prop:ultrahot} and \Cref{prop:moderate:cold}) and shows that our results agree with classical results on the Gaussian REM. \Cref{sec:proof:ultrahigh} proves \Cref{prop:ultrahot}, the REM result in the ultra-high temperature regime. \Cref{sec:proof:moderate:cold} proves \Cref{prop:moderate:cold}, the REM result in the moderate-to-cold temperature regime. \Cref{sec:proof:main} combines these results to prove our main results, i.e.~\Cref{thm:pointwise,thm:top:Lyapunov}. 

\section{Reduction to the random energy model} \label{sec:reduction:REM}

Our key observation is that the analysis of the top Lyapunov exponent $\log s_1(X)$ can be reduced to the log-partition function of a suitable random energy model. The first step is to note that, by an argument analogous to Proposition 8.1 of \cite{hanin2021non}, which studies one matrix product $X_1$, we can approximate the normalized top Lyapunov exponent of $X$ by the effect of $X$ on a fixed vector $\theta \in \cS^{n-1}$:
\begin{align*}
     \msup_{\theta' \in \cS^{n-1}} \mfrac{1}{N} \log \| X \theta' \|
    \;\approx\; 
    \mfrac{1}{N} \log \| X \theta \|
     \;.
     \tagaligneq \label{eq:neglect:pointwise}
\end{align*}
This is made formal by \Cref{lem:sup:to:pointwise} below. Next, observe that we can express 
\begin{align*}
    \| X \theta \|^2 
    \;=\;
    \theta^\top \Big(\mfrac{1}{m} \msum_{i, j \leq m} X_i^\top X_j \Big) \theta 
    \;=\;
    \theta^\top \Big(\mfrac{1}{m} \msum_{i, j \leq m} X_{i, N-1:1}^\top X_{iN}^\top X_{jN} X_{j, N-1:1} \Big) \theta 
    \;,
\end{align*} 
where we have denoted
\begin{align*}
    X_{i,N-1:1} \;\coloneqq\; X_{i(N-1)} \cdots X_{i1}\;.
\end{align*}
Conditioning on $(X_{i,N-1:1})_{i \leq m}$, the quantity $\| X \theta\|^2$ can be viewed as a random quadratic form in the i.i.d.~Gaussian matrices $(X_{iN})_{i \leq m}$. In \Cref{lem:Hanson:Wright}, we will use a concentration inequality over the randomness of $(X_{iN})_{i \leq m}$ to show that we can approximate 
\begin{align*}
    \| X \theta \|^2  \;\approx\; \theta^\top \Big(\mfrac{1}{m} \msum_{i \leq m} X_{i,N-1:1}^\top X_{i,N-1:1} \Big) \theta 
    \;=\;
    \mfrac{1}{m} \msum_{i \leq m}  \| X_{i,N-1:1} \theta  \|^2\;.
    \tagaligneq \label{eq:Xtheta:REM}
\end{align*}
\eqref{eq:Xtheta:REM} now involves $m$ i.i.d.~univariate quantities, each involving one $N-1$ matrix product. In particular, the distribution of each summand has been completely characterized by \cite{hanin2021non}. To make this formal, consider a collection of i.i.d.~random variables $(Y_{ij})_{1 \leq i \leq m, 1 \leq j \leq N-1}$ each distributed as 
\begin{align*}
    Y_{ij} \;\sim\; \mfrac{1}{2} \log\Big( \mfrac{1}{n} \chi^2_n \Big)\;,
\end{align*}
where $\chi^2_n$ is a chi-squared random variable with $n$ degrees of freedom. The following holds:

\begin{lemma}[Special case of Lemma 9.5 of \cite{hanin2021non}] \label{lem:vector:norm:to:chi:sq} For any fixed $\theta \in \cS^{n-1}$, $\| X_{1,N-1:1} \theta \|$ is identically distributed as 
$\exp\big( \sum_{j=1}^{N-1} Y_{1j} \big)$.
\end{lemma}

Since $X_{i,N-1:1}$ are i.i.d.~across $1 \leq i \leq m$, \Cref{lem:vector:norm:to:chi:sq} implies that for a fixed $\theta \in \cS^{n-1}$, the collection $(  \| X_{1,N-1:1} \theta \|, \ldots,   \| X_{m,N-1:1} \theta \| )$ is identically distributed as 
\begin{align*}
    \Big( \exp\big(\msum_{j=1}^{N-1} Y_{1j}\big) \,,\, \ldots \,,\, \exp\big(\msum_{j=1}^{N-1} Y_{mj}\big) \Big)\;.
\end{align*}
This allows us to express 
\begin{align*}
    \mfrac{1}{2N} \log\Big( \mfrac{1}{m} \msum_{i \leq m}  \| X_{i,N-1:1} \theta  \|^2 \Big)
    \;\overset{d}{=}\;
    \mfrac{1}{2N} \log\Big( 
        \mfrac{1}{m} \msum_{i \leq m} 
    e^{2 \sum_{j=1}^{N-1} Y_{ij} }
    \Big)
    \;.
    \tagaligneq \label{eq:REM:first:derivation}
\end{align*}
Up to shifting and rescaling, this can be interpreted as the log partition function of a random energy model with non-Gaussian energies $\{ \sum_{j=1}^{N-1} Y_{ij}  \}_{i \leq m}$; see \Cref{sec:REM}.

\vspace{.5em}

The rest of this section makes the above argument formal. For the removal of the supremum, we note that for a single matrix product, Proposition 8.1 of \cite{hanin2021non} establishes the approximation 
\begin{align*}
    \sup_{\theta' \in \cS^{n-1}} \mfrac{1}{N} \log \| X_1 \theta' \| \;\approx\; \mfrac{1}{N} \log \| X_1 \theta\|
\end{align*}
for any fixed $\theta \in \cS^{n-1}$. It turns out that their proof technique directly extends to our setting of a general $m$.

\begin{lemma} \label{lem:sup:to:pointwise} There exists a universal constant $C > 0$ such that, for any $\epsilon \in (0,1)$ and $\theta \in \cS^{n-1}$, we have
\begin{align*}
    \P\Big( 
        \Big| 
            \mfrac{1}{N}
            \log \| X \theta \|
            - 
            \mfrac{1}{N}
            \sup_{\theta' \in \cS^{n-1}} 
            \log \| X \theta' \|
        \Big|  
        \,\geq\,
        \mfrac{1}{2 N} \log \Big( \mfrac{n}{\epsilon^2} \Big)
    \Big)
    \leq 
    (C \epsilon)^{1/2}
    \;.
\end{align*}
\end{lemma}

\begin{proof}[Proof of \Cref{lem:sup:to:pointwise}] The result follows verbatim from the proof of Proposition 8.1 of \cite{hanin2021non}: The only property of $X_1 = X_{1N} \cdots X_{11}$ used in their proof is that $X_1$ is distributionally invariant under right multiplication by a Haar orthogonal matrix in $\R^{n \times n}$, which holds also for the sum of Gaussian matrix products $X = \frac{1}{\sqrt{m}} \sum_{i \leq m} X_i$.
\end{proof}

\vspace{.5em}

We now establish a concentration inequality over the randomness of $X_{iN}$ for the approximation \eqref{eq:Xtheta:REM}.

\begin{lemma} \label{lem:Hanson:Wright}  There exists some universal constant $c > 0$ such that for every $\epsilon > 0$ and $\theta \in \R^n$, 
\begin{align*}
    &\;
    \P\Big( 
        \,
        \Big|  
            \| X \theta \|^2
            -
            \Big(
                \mfrac{1}{m} \msum_{i \leq m} 
                \| X_{i,N-1:1} \theta \|^2
            \Big)
        \Big| 
        > \epsilon \, 
        \Big(
                \mfrac{1}{m} \msum_{i \leq m} 
                \| X_{i,N-1:1} \theta \|^2
            \Big)
    \Big)
    \;\leq\;
    2
    e^{
        - c n  \min\{ \epsilon^2, \epsilon\}
    }
    \;.
\end{align*}
\end{lemma}

\begin{proof}[Proof of \Cref{lem:Hanson:Wright}] First denote $\varphi_i \coloneqq X_{i, N-1:1} \theta$. Conditioning on $(\varphi_i)_{i \leq m}$, the vector of concern is an empirical average of Gaussian vectors 
\begin{align*}
    X \theta \;=\; \mfrac{1}{\sqrt{m}} \msum_{i \leq m} X_i \theta \;=\; \mfrac{1}{\sqrt{m}} \msum_{i \leq m} X_{iN} \varphi_i \;,
\end{align*}
where 
\begin{align*}
    X_{iN} \varphi_i  \,|\, \varphi_i \;\sim\; \cN\Big( 0,  \mfrac{1}{n} \|\varphi_i\|^2 \ind_n \Big)
    \;\equiv\;
     \cN\Big( 0,  \mfrac{1}{n} \| X_{i,N-1:1} \theta\|^2 \ind_n \Big)
    \;.
\end{align*}
Therefore 
\begin{align*}
    X \theta \,|\, (\varphi_i)_{i \leq m}
    \;\sim\; 
    \cN\Big( 0, \mfrac{1}{n} \kappa^2_m \ind_n \Big)
    \;,
    \quad 
    \text{ where we write }
    \kappa^2_m \;\coloneqq\; \mfrac{1}{m} \msum_{i \leq m}  \| X_{i,N-1:1} \theta\|^2\;.
\end{align*}
By the Hanson-Wright inequality (see e.g.~Theorem 6.2.1.~of \cite{vershynin2018high}), there exists some universal constant $c > 0$ such that for every $\epsilon > 0$, almost surely
\begin{align*}
    \P\Big( 
        \big| 
            \| X \theta \|^2
            -
            \kappa^2_m
        \big| 
        > \epsilon
        \,\Big|\, (\varphi_i)_{i \leq m}
    \Big)
    &\;\leq\;
    2
    \exp\Big( 
        - c
        \min\Big\{ 
            \mfrac{n^2 \epsilon^2}{ \kappa_m^4 \| \ind_n \|_{\rm F}^2 }
            \,,\,    
            \mfrac{n \epsilon}{\kappa_m^2 \| \ind_n  \|_{op} }
        \Big\}
    \Big)
    \\
    &\;=\;
    2
    \exp\Big( 
        - c n \min\Big\{  \mfrac{ \epsilon^2}{ \kappa^4_m } \,,\, \mfrac{ \epsilon}{\kappa_m^2} \Big\}
    \Big)
    \;.
\end{align*}
To obtain the required bound, we rescale $\epsilon$ by $\kappa_m^2$ and take expectation on both sides of the inequality above. 
\end{proof}

\vspace{.5em}

\Cref{lem:Hanson:Wright} implies a control on $\log \| X \theta \|$:

\begin{corollary} \label{cor:log:Hanson:Wright}  There exists some universal constant $c > 0$ such that for every $\epsilon \in (0,1)$ and $\theta \in \R^n \setminus \{0\}$, 
\begin{align*}
    \P\Big( 
        \Big|  
            \log \| X \theta \|
            -
            \mfrac{1}{2} \log  
            \Big(
                \mfrac{1}{m} \msum_{i \leq m} 
                \| X_{i,N-1:1} \theta \|^2
            \Big)
        \Big| 
        \;\geq\; 
        \mfrac{\epsilon}{2(1-\epsilon)}
    \Big)
    \;\leq\;
    2 e^{-c n \epsilon^2 }
    \;.
\end{align*}
    
\end{corollary}

\begin{proof}[Proof of \Cref{cor:log:Hanson:Wright}] Denote
\begin{align*}
    \Delta \;\coloneqq\; \mfrac{\| X \theta \|^2 - \frac{1}{m} \sum_{i \leq m} 
            \| X_{i,N-1:1} \theta \|^2 }{\frac{1}{m} \sum_{i \leq m} 
            \| X_{i,N-1:1} \theta \|^2}\;.
\end{align*}
For $\epsilon \in (0,1)$, \Cref{lem:Hanson:Wright} then reads $\P( |\Delta | > \epsilon ) \leq 2 e^{-c n \epsilon^2 }$.  Note also that $\frac{x}{1+x} \leq \log(1+x) \leq x$ for all $x > -1$ and therefore $|\log(1+x)| \leq \frac{|x|}{|1-|x||}$. Then with probability at least $1- 2 e^{-c n \epsilon^2 }$, we have
\begin{align*}
     \Big|  
        \log \| X \theta \|
        -
        \mfrac{1}{2} \log  
        \Big(
            \mfrac{1}{m} \msum_{i \leq m} 
            \| X_{i,N-1:1} \theta \|^2
        \Big)
    \Big| 
    \;=&\;
    \mfrac{1}{2}
    \,
    \big|   
    \log 
        (1 + \Delta)
    \big| 
    \;\leq\;
    \mfrac{\epsilon}{2(1-\epsilon)}\;.
\end{align*}
\end{proof}

Combining this with \Cref{lem:sup:to:pointwise} by the triangle inequality and applying \eqref{eq:REM:first:derivation}, we obtain the desired approximation that 
\begin{align*}
    \mfrac{1}{N} \sup_{\theta' \in \cS^{n-1}} \log \| X \theta' \|
    \;\approx\;
    \mfrac{1}{2N} \log  
    \Big(
        \mfrac{1}{m} \msum_{i \leq m} 
        \| X_{i,N-1:1} \theta \|^2
    \Big)
    \;\overset{d}{=}\;
    \mfrac{1}{2N} \log\Big( 
        \mfrac{1}{m} \msum_{i \leq m} 
    e^{2 \sum_{j=1}^{N-1} Y_{ij} }
    \Big)
    \;.
    \tagaligneq \label{eq:empirical:REM:approx}
\end{align*}

\section{Log-partition function of a non-Gaussian REM}
\label{sec:REM}

We now focus on the quantity
\begin{align*}
    \cE 
    \;\coloneqq\;
    \log\Big( 
        \mfrac{1}{m} \msum_{i \leq m} 
    e^{2 \sum_{j=1}^{N-1} Y_{ij} }
    \Big)
    \;,
\end{align*}
where the i.i.d.~random variables $(Y_{ij})_{1 \leq i \leq m, 1 \leq j \leq N-1}$ are distributed as $\frac{1}{2} \log\big( \frac{1}{n} \chi^2_n \big)$ and $\chi^2_n$ is a chi-squared random variable with $n$ degrees of freedom. $\cE$ is related to the log-partition function of a non-Gaussian random energy model. Specifically, by standard estimates of the mean and variance of a log-chi-squared variable with growing degrees of freedom (see \Cref{lem:log:chi:sq}), we have
\begin{align*}
     \mean[Y_{11}]
    \;=&\;
    \mean\Big[
        \mfrac{1}{2} 
        \log\Big( \mfrac{1}{n} \chi^2_n \Big)
    \Big]
    \;=\;
    - \mfrac{1}{2n} + O\Big( \mfrac{1}{n^2} \Big)
    \;,
    \\
    \Var[Y_{11}]
    \;=&\;
        \Var\Big[
            \mfrac{1}{2} 
            \log\Big( \mfrac{1}{n} \chi^2_n \Big)
        \Big] 
    \;=\;
    \mfrac{1}{2n} + O\Big( \mfrac{1}{n^2} \Big)
    \;.
\end{align*} 
Writing
\begin{align*}
    \mu
    \;\coloneqq&\;
    - \mfrac{1}{2n}
    &\text{ and }&&
    \sigma 
    \;\coloneqq&\;
    \mfrac{1}{\sqrt{2n}},
\end{align*}
we can define the (asymptotically) standardized random variables, 
\begin{align*}
    E_i 
    \;\coloneqq\; 
    - \mfrac{\sum_{j=1}^{N-1} Y_{ij} - (N-1) \mu}{  \sigma \, \sqrt{N-1}}\;.
    \tagaligneq \label{eq:REM:energy:defn}
\end{align*}
These correspond to the i.i.d.~random energies, and allow us to express
\begin{align*}
    \cE
    \;=&\;
    \log \Big( \mfrac{1}{m} \msum_{i \leq m} e^{ - 2 \sigma \sqrt{N-1} \, E_i} \Big)
    +
    2 (N-1)  \mu 
    \\
    \;=&\;
    \log \Big( \mfrac{1}{m} \msum_{i \leq m} e^{ - \beta \sqrt{\log m} \, E_i} \Big)
    +
    2 (N-1)  \mu 
    \;,
    \tagaligneq \label{eq:expression:Ek}
\end{align*}
which is a shifted log-partition function of a random energy model with $m$ different configurations, random energies $(E_i)_{i \leq m}$, and the inverse temperature parameter
\begin{align*}
    \beta 
    \;=\;
    \mfrac{2 \sigma \sqrt{N-1}}{\sqrt{\log m}}
    \;=\;
    \mfrac{\sqrt{2 (N-1)}}{\sqrt{n \log m}}
    \;.
\end{align*}

\vspace{.5em}

Although $E_i$'s are non-Gaussian, we will show that they can be approximated by Gaussian energies for the purpose of studying the limit of $\cE$ as $n, N, m$ grow. To obtain tight approximation errors, we require a non-uniform Gaussian approximation bound (see \Cref{prop:nonuniform:CLT:crude}). Then, following known results for the Gaussian REM \citep{derrida1981random,dorlas2001large,bovier2002fluctuations}, we may conjecture that asymptotically,
\begin{align*}
    \cE 
    \;\approx\; 
    \bar \cE
    \;\coloneqq\;
    \begin{cases}
            2(N-1) \mu + \mfrac{\beta^2 \log m}{2}
            &
            \text{ if }
            \beta \leq \sqrt{2}
            \;,
            \\
            2(N-1) \mu + \sqrt{2} \, \beta \log m - \log m
            & 
            \text{ if } 
            \beta > \sqrt{2}
            \;.
        \end{cases}
        \tagaligneq \label{eq:REM:conj}
\end{align*}
The energy approximation $\bar \cE$ exhibits a phase transition at $\beta=\sqrt{2}$. Moreover, since $2(N-1)\mu = -\frac{N-1}{n} = - \frac{\beta^2 \log m}{2}$, we can express 
\begin{align*}
    \bar \cE
    \;=&\;
    \begin{cases}
            0
            &
            \text{ if }
            \beta \leq \sqrt{2}
            \;,
            \\
            - \frac{(\beta - \sqrt{2})^2}{2} \, \log m
            & 
            \text{ if } 
            \beta > \sqrt{2}
            \;,
    \end{cases}
    \\
    \;=&\;
    \begin{cases}
            0
            &
            \text{ if }
            (N-1)/n \,\leq\, \log m
            \;,
            \\
            - \big(\sqrt{(N-1) / n} \,-\, \sqrt{\log m}\big)^2
            & 
            \text{ if } 
            (N-1)/n \,>\, \log m
            \;,
    \end{cases}
    \\
    \;=&\;
    2 Z\;,
\end{align*}
where $Z$ is defined in \eqref{eq:limit} in the introduction. Compared to existing works, the additional difficulty in our setting is that $\beta$ is no longer fixed but depends on $n$, $N$ and $m$, and may asymptotically vanish or diverge.

\vspace{.5em}

We now state concentration inequalities that make \eqref{eq:REM:conj} precise. The first result concerns the ultra-high temperature regime with $\beta=o(1)$: In this case, all energies contribute equally, and applying Markov's inequality to an i.i.d.~average $\frac{1}{m} \sum_{i \leq m} e^{- \beta\sqrt{\log m} E_i}$ suffices for computing the limit. Note that in this case $\bar \cE = 0$.

\begin{proposition} \label{prop:ultrahot} Assume $\beta = o(1)$. Then there exist universal constants $C, c > 0$ such that, for all $\epsilon \in (0,1)$,
\begin{align*}
    \P\Big( 
        \,
        |\cE - \bar \cE |
        \;\geq\;  
        C
        \Big(
            \mfrac{\epsilon}{(1-\epsilon) }    
            +
            \mfrac{\beta^2 \log m}{n}
        \Big)
    \Big)
    \;\leq\; 
    \epsilon^{-2} 
    e^{ - c \log m }\;.
\end{align*}
\end{proposition}

\vspace{.5em}

The next result concerns the moderate-to-low temperature regime with $\beta = \Omega(1)$: In this case, not all energies contribute equally and the empirical average $\frac{1}{m} \sum_{i \leq m} e^{- \beta\sqrt{\log m} E_i}$ is dominated by the outlier values. To capture the effects of these outliers, we adapt the approach of \cite{dorlas2001large} that studies Gaussian REM with Varadhan's lemma. The key differences are that we employ Laplace's method to explicitly compute the error bounds and accommodate the non-Gaussian (though approximately Gaussian) energies $E_i$'s.

\begin{proposition} \label{prop:moderate:cold} Assume that $\beta= \Omega(1)$, $\log m = o(N^{1/3})$ and $N = o(n^3)$. Then there exist some universal constants $C, c > 0$ such that for every $\epsilon \in (0,1)$,
\begin{align*}
    \P\Big( 
         \mfrac{|\cE - \bar \cE|}{\beta^2 \log m}  
        \;>\;
        C \Big( \mfrac{\epsilon}{(1-\epsilon) \beta \log m}
            +
            \mfrac{1}{\beta (\log m)^{1/4}} \Big)
    \Big)
    \;\leq\; 
    (1+ \epsilon^{-2}) e^{ - c  \, (\log m)^{3/4} }\;.
\end{align*}
\end{proposition}

\subsection{Related literature on REM} \label{sec:literature:REM}

The Gaussian random energy model (REM), introduced and solved by \cite{derrida1981random}, is one of the canonical models in statistical physics that is attractive for its exact solvability. The central takeaway from the REM is its freezing phase transition: Above a critical temperature threshold, the free energy is governed by a massive number of configurations that have ``typical'' energies, whereas below that temperature, the free energy is dominated by a small number of configurations with excessively low energies. 

\vspace{.5em}

Our limiting expression, $\bar \cE$, agrees with known Gaussian REM limits. To see this, recall that \cite{derrida1981random} studies the normalized average free energy
\begin{align*}
    -
    \mfrac{\tilde T}{\tilde N}
    \,
    \mean\Big[
    \log 
    \Big( 
        \msum_{i=1}^{2^{\tilde N}} \exp\Big(
            -  \mfrac{\sqrt{\tilde N} \, \tilde J}{\sqrt{2} \, \tilde T} \eta_i \Big)
    \Big)
    \Big]\;,
    \tagaligneq \label{eq:energy:Derrida}
\end{align*}
where $\tilde T$ is the temperature, $\tilde N$ is the system size, $\tilde J$ is a scale parameter and $\eta_i$'s are i.i.d.~standard Gaussians. \cite{derrida1981random} shows that it converges to
\begin{align*}
    \begin{cases}
         - \tilde T \log 2 - \mfrac{(\tilde J)^2}{4 \tilde T}
         &\text{ if } \tilde T > \frac{\tilde J}{2 \sqrt{\log 2}}\;,
         \\
        - \tilde J \sqrt{\log 2}
         &\text{ if } \tilde T < \frac{\tilde J}{2 \sqrt{\log 2}}\;.
    \end{cases}
    \tagaligneq \label{eq:limit:Derrida}
\end{align*}
It is also well-established that the normalized free energy, i.e.~the random variable representing the log-partition function without the expectation in \eqref{eq:energy:Derrida}, converges to \eqref{eq:limit:Derrida} almost surely (see \cite{bovier2002fluctuations,dorlas2001large}). By a reparameterization with $\tilde N = \frac{\log m}{\log 2}$ and $\tilde J / \tilde T = \sqrt{2 \log 2} \,  \beta$, \eqref{eq:energy:Derrida} and \eqref{eq:limit:Derrida} are equivalent to the statement that
\begin{align*}
    -
    \mfrac{\tilde T \log 2}{\log m }
        \mean\Big[
        \log 
        \Big( 
            \msum_{i=1}^{m} e^{ - \beta \sqrt{\log m} \,  \eta_i}
        \Big)
        \Big]
    \;\overset{\tilde N \rightarrow \infty}{\rightarrow}\;
    \begin{cases}
         - \tilde T \log 2 
         - \tilde T \mfrac{\log 2}{2} \beta^2
         &\text{ if }  \beta < \sqrt{2}\;,
         \\
        - \tilde T \sqrt{2} \, (\log 2) \beta
         &\text{ if }   \beta > \sqrt{2}\;.
    \end{cases}
\end{align*}
Dividing across by $ - \tilde T \log 2$ followed by a subtraction by $1$, the above gives 
\begin{align*}
    \mfrac{1}{\log m }
        \mean\Big[
        \log 
        \Big( 
            \mfrac{1}{m} \msum_{i=1}^{m} e^{ - \beta \sqrt{\log m} \,  \eta_i}
        \Big)
        \Big]
    \;\overset{\tilde N \rightarrow \infty}{\rightarrow}\;
    \begin{cases}
         \mfrac{\beta^2}{2} 
         &\text{ if }  \beta < \sqrt{2}\;,
         \\
        \sqrt{2} \beta - 1
         &\text{ if }   \beta > \sqrt{2}\;.
    \end{cases}
\end{align*}
Rescaling both sides by $\log m$ and adding $2(N-1) \mu = - \frac{\beta^2 \log m}{2}$ to both sides gives an approximation that agrees with our result with $\cE \approx \bar \cE$.

\vspace{.5em}

Our notion of convergence differs from that in the typical REM literature. Indeed, observe that \Cref{prop:ultrahot,prop:moderate:cold} only imply convergence in probability of the log-partition function $\cE$. To reconcile this with the almost sure convergence in the REM literature \citep{dorlas2001large,bovier2002fluctuations}, we note that in those works, one takes $m=2^M$, where $M$ is the number of spins and $m$ is the number of configurations, and considers the limiting behaviour as $M \rightarrow \infty$. This is equivalent to considering the limiting behaviour along a subsequence $\{ m : m = 2^M \text{ for some } M \in \N\}$ compared to our limit $m \rightarrow \infty$. Indeed, setting $m=2^M$ in both \Cref{prop:ultrahot,prop:moderate:cold} gives probability bounds that decay at an exponential or stretched-exponential rate in $M$, which implies almost sure convergence by the Borel-Cantelli lemma.

\section{Proof of \Cref{prop:ultrahot}: Ultra-high temperature regime} \label{sec:proof:ultrahigh}

This section proves \Cref{prop:ultrahot}, which concerns the result in the ultra-high temperature regime $\beta = o(1)$. Recall that our object of interest is 
\begin{align*}
    \cE \;=\;
        \log\Big( 
            \mfrac{1}{m} \msum_{i \leq m} 
        e^{2 \sum_{j=1}^{N-1} Y_{ij} }
        \Big)
    \;,
\end{align*}
where the i.i.d.~random variables $(Y_{ij})_{1 \leq i \leq m, 1 \leq j \leq N-1}$ are distributed as $\frac{1}{2} \log\big( \frac{1}{n} \chi^2_n \big)$ and $\chi^2_n$ is a chi-squared random variable with $n$ degrees of freedom. We first recall some standard properties of the log-chi-squared distribution:

\begin{lemma}[Properties of the log-chi-squared distribution] \label{lem:log:chi:sq} For $\nu > -\frac{n}{2}$ with $\nu = o(n)$, we have
\begin{align*} 
    &\;
    \mean[Y_{11}]
    \;=\;
    - 
    \mfrac{1}{2n} 
    +
    O\Big( \mfrac{1}{n^2} \Big)
    \;,
    \qquad 
    \Var[Y_{11}]
    \;=\;
    \mfrac{1}{2n}
    + 
    O\Big( \mfrac{1}{n^2} \Big)
    \;,
    \\
    &\;\mean\Big[ e^{2 \nu \sum_{j=1}^{N-1} Y_{1j}}  \Big]
    \;=\;
    \exp\Big(
        \mfrac{(\nu^2 - \nu)(N-1)}{n}
        +
        O\Big( 
            \mfrac{(N-1) (1+|\nu|^3)}{n^2} 
        \Big)
    \Big) 
    \;.
\end{align*}
\end{lemma}

\begin{proof}[Proof of \Cref{lem:log:chi:sq}] To compute the moments, let $\psi$ denote the digamma function and recall that for a Gamma random variable $\phi \sim \Gamma(\alpha, \theta)$ with density $\frac{x^{\alpha-1}}{\Gamma(\alpha) \, \theta^\alpha} e^{-x/\theta}$, 
\begin{align*}
    \mean[\log \phi ] \;=&\; \psi(\alpha) + \log \theta
    &\text{ and }&&
    \Var[\log \phi] \;=&\; \partial \psi(\alpha)
    \;.
\end{align*}
Also note that $n e^{2 Y_{11}} \overset{d}{=} \chi^2_n \sim \Gamma(\frac{n}{2},2)$. By the standard approximation of a polygamma function, we have
\begin{align*}
    \mean[Y_{11}]
    \;=\;
    \mean\Big[ \mfrac{1}{2} \log\Big( \mfrac{1}{n} \chi^2_n \Big) \Big]
    \;=&\;
    \mfrac{1}{2} \Big( \psi\Big( \mfrac{n}{2} \Big) + \log 2  - \log n \Big)
    \\
    \;=&\;
    \mfrac{1}{2} 
    \Big( 
        \log\Big( \mfrac{n}{2} \Big) 
        -
        \mfrac{1}{n}
        +
        O\Big( \mfrac{1}{n^2} \Big)
        +
        \log 2 - \log n
    \Big)
    \\
    \;=&\;
    -\mfrac{1}{2n}
    +
    O\Big( \mfrac{1}{n^2} \Big)
    \;,
\end{align*}
as well as 
\begin{align*}
    \Var[Y_{11}]
    \;=\; 
    \mfrac{1}{4} \Var\Big[  \log\Big( \mfrac{1}{n} \chi^2_n \Big)  \Big]
    \;=\;
    \mfrac{1}{4} \,\partial \psi\Big( \mfrac{n}{2} \Big) 
    \;=\;
    \mfrac{1}{2n}
    +
    O\Big( \mfrac{1}{n^2} \Big)
    \;.
\end{align*}
To prove the m.g.f.~formula, let $\Gamma$ be the Gamma function. For $\nu > -\frac{n}{2}$, by independence and moment formulas for chi-squared variables, we have
\begin{align*}
        \mean\Big[ e^{2 \nu \sum_{j=1}^{N-1} Y_{1j}}  \Big]
        \;=&\;
        \big( \mean[ e^{2 \nu  Y_{11} } ] \big)^{N-1}
        \;=\;
        \Big( 
            \mean\Big[ \Big(\mfrac{1}{n} \chi^2_n  \Big)^\nu \Big] 
        \Big)^{N-1}
        \\
        \;=&\;
        \Big( \mfrac{2}{n} \Big)^{\nu (N-1)}
        \Big( 
            \mfrac{\Gamma\big( \frac{n}{2} + \nu \big)}{\Gamma\big( \frac{n}{2} \big)}
        \Big)^{N-1}
        \\
        \;=&\;
        \exp\Big( 
            - 
            \nu (N-1) \log \mfrac{n}{2}
            +
            (N-1) 
            \Big( 
                \log \Gamma\Big( \mfrac{n}{2} + \nu \Big)
                -
                \log \Gamma\Big( \mfrac{n}{2} \Big)
            \Big) 
        \Big)
        \;.
\end{align*}
Consider the asymptotic expansion
\begin{align*} 
    &\;
    \log \Gamma(w) \,=\, \Big( w - \mfrac{1}{2}\Big) \log(w)  - w +  \mfrac{1}{2} \log(2\pi)
    +
    \mfrac{1}{12w}
    +
    O( |w|^{-3})
    \qquad
    \text{ as } |w| \rightarrow \infty
\end{align*}
which holds uniformly on sets of the form $|\arg(w)\,| < \pi-\delta$ for a fixed $\delta > 0$. Noting that $\nu = o(n)$, we obtain 
\begin{align*}
    \mean\Big[ &\, e^{2 \nu \sum_{j=1}^{N-1} Y_{1j}}  \Big]
    \\
    &\;=\;
    \exp\Big( 
            - 
            \nu (N-1) \log \mfrac{n}{2}
            +
            O\Big( \mfrac{N}{n^3} \Big)
    \\
    &\hspace{4em}
            +
            (N-1) 
            \Big( 
                \Big( \mfrac{n-1}{2} + \nu \Big) \log\Big(\mfrac{n}{2} + \nu \Big)
                -
                \mfrac{n-1}{2} \log \mfrac{n}{2} 
                -
                \nu
                + 
                \mfrac{1}{6 n + 12\nu }
                -
                \mfrac{1}{6 n }
            \Big) 
        \Big)
    \\
    \;=&\;
    \exp\Big(
            (N-1) 
            \Big(
                \nu \log\Big( 1 + \mfrac{2 \nu}{n} \Big)
                + 
                \mfrac{n-1}{2} \log\Big(1 + \mfrac{2\nu}{n}  \Big)
                - 
                \nu  
            \Big) 
            \,+\,
            O\Big( \mfrac{N(1+|\nu|)}{n^2} \Big)
        \Big)
    \\
    \;=&\;
    \exp\Big(
             \mfrac{2 \nu^2 (N-1)}{n} 
                + 
                \mfrac{(N-1)(n-1)}{2} \Big( \mfrac{2\nu}{n} - \mfrac{2 \nu^2}{n^2} \Big) 
                - 
                \nu (N-1) 
            \,+\,
            O\Big( \mfrac{N(1+|\nu|^3)}{n^2} \Big)
        \Big)
    \\
    \;=&\;
    \exp\Big(
             \mfrac{(\nu^2 - \nu) (N-1)}{n} 
                +
            O\Big( \mfrac{N (1+|\nu|^3)}{n^2} \Big)
        \Big)
    \;.
\end{align*}
\end{proof}

\vspace{.5em}

The proof of \Cref{prop:ultrahot} follows from applying the Markov inequality on an empirical average of $m$ i.i.d.~random variables $e^{- \beta \sqrt{\log m} E_i}$.

\begin{proof}[Proof of \Cref{prop:ultrahot}] We first compute the mean and variance of these random variables: For $\nu > - \frac{n}{2}$, \Cref{lem:log:chi:sq} allows us to compute
\begin{align*}
    \mean\big[ e^{- \nu \beta \sqrt{\log m} E_1} \big]
    \;\overset{\eqref{eq:expression:beta}}{=}&\;
    \mean\big[ e^{- 2 \nu \sigma \sqrt{N-1} \,  E_1} \big]
    \\
    \;\overset{\eqref{eq:REM:energy:defn}}{=}&\;
    e^{ - 2 \nu (N-1) \mu}
    \,
    \mean\big[ 
        e^{ 2 \nu \sum_{j=1}^{N-1} Y_{1j} }
    \big]
    \\
    \;=&\;
    \exp\Big(  \mfrac{\nu (N-1)}{n} + \mfrac{(\nu^2 - \nu) (N-1)}{n}
        +
        O\Big( 
            \mfrac{N (1+|\nu|^3)}{n^2} 
        \Big)  \Big)
    \\
    \;=&\;
    \exp\Big(  \mfrac{\nu^2  (N-1)}{n}
        +
        O\Big( 
            \mfrac{N (1+|\nu|^3)}{n^2} 
        \Big)  \Big)
    \\
    \;=&\;
    \exp\Big(  \mfrac{\nu^2  \beta^2 \log m}{2}
        +
        O\Big( 
            \mfrac{\beta^2 (1+|\nu|^3) \log m}{n} 
        \Big)  
    \Big)
    \;.
\end{align*}
This implies 
\begin{align*}
    \mean\big[ e^{- \beta \sqrt{\log m}\, E_1} \big]
    \;=&\;
    e^{\frac{\beta^2 \log m}{2} + O\big(\frac{\beta^2 \log m}{n}\big)}
    \;,
    \tagaligneq \label{eq:REM:mean:compute}
    \\[1em]
    \Var\big[ e^{- \beta \sqrt{\log m}\, E_1} \big]
    \;=&\;
    \mean\big[ e^{- 2\beta\sqrt{\log m}\, E_1}  \big]
    -
    \mean\big[ e^{- \beta\sqrt{\log m}\, E_1}  \big]^2
    \\
    \;=&\;
    \big( 
        e^{2 \beta^2 \log m} - e^{\beta^2 \log m}
    \big) 
    \,
    e^{O\big(\frac{\beta^2 \log m}{n}\big)}
    \;,
    \tagaligneq \label{eq:REM:var:compute}
\end{align*}
and therefore 
\begin{align*}
    \mean\big[ e^{- \beta\sqrt{\log m}\, E_1} \big]^{-2}
    \,
    \Var\big[ e^{- \beta \sqrt{\log m}\, E_1} \big]
    \;=&\; 
    \big( e^{\beta^2 \log m} - 1 \big) \,
    e^{ O( \beta^2 \log m  ) }
    \;.
\end{align*}
Now recall that 
\begin{align*}
    e^{\cE - 2 (N-1) \mu} 
    \;=\;
    \mfrac{1}{m} \msum_{i \leq m} e^{ - \beta\sqrt{\log m}\, E_i}
    \;.
\end{align*}
For $\epsilon  \in (0,1)$, consider the event 
\begin{align*}
    A_\epsilon \;\coloneqq\;
    \Big\{ 
        \big|e^{\cE - 2 (N-1) \mu}  - \mean\big[ e^{ - \beta \sqrt{\log m}\, E_1} \big] \big| 
        \;\leq\;
        \epsilon 
        \,
        \mean\big[ e^{ - \beta \sqrt{\log m}\, E_1} \big]
    \Big\}
    \;.
\end{align*}
By Markov's inequality and independence of $E_i$'s, we get that for any $\epsilon \in (0,1)$,
\begin{align*}
    1 - \P( A_\epsilon )
    \;\leq&\;
    \mfrac{1}{m \epsilon^2} \, \big( e^{\beta^2 \log m} - 1 \big) \, e^{O(\beta^2 \log m) }
    \\
    \;\leq&\;
    \epsilon^{-2} \, e^{ -  (1  + o(1) ) \log m }
    \;,
\end{align*}
where we have used the assumption that $\beta= o(1)$. On the event $A_\epsilon$, recalling that $-\frac{\beta^2\log m}{2} = 2(N-1)\mu$ and $\bar \cE=0$ in the case $\beta = o(1)$, we have that for all $\epsilon \in (0,1)$,
\begin{align*}
    &\;|\cE - \bar \cE| 
    \;=\;
    \Big| 
        \cE - 2 (N-1) \mu  - \mfrac{\beta^2 \log m}{2}
    \Big|
    \;=\; 
    \Big| 
        \log \, \exp (\cE - 2 (N-1) \mu)   - \mfrac{\beta^2 \log m}{2} 
    \Big|
    \\
    &\;=\;
    \Big| 
        \log  
        \Big( 
            1
            +
            \mfrac{
                e^{\cE - 2 (N-1) \mu}
                -
                 \mean\big[ e^{ - \beta\sqrt{\log m}\, E_1} \big] 
            }
            {
                 \mean\big[ e^{ - \beta\sqrt{\log m}\, E_1} \big] 
            }
        \Big)
        +
        \log  \mean\big[ e^{ - \beta\sqrt{\log m}\, E_1} \big]
        - 
        \mfrac{\beta^2 \log m}{2} 
    \Big|
    \\
    &\;\leq\;
    \Big| 
    \log 
        \Big( 
            1
            +
            \mfrac{
                e^{\cE - 2 (N-1) \mu}
                -
                 \mean\big[ e^{ - \beta\sqrt{\log m}\,E_1} \big] 
            }
            {
                 \mean\big[ e^{ - \beta\sqrt{\log m}\,E_1} \big] 
            }
        \Big)
    \Big|
    +
    \Big| 
        \log  \mean\big[ e^{ - \beta\sqrt{\log m}\,E_1} \big] 
        - 
        \mfrac{\beta^2 \log m}{2} 
    \Big| 
    \\
    &\;\leq\;
    \max\{ | \log(1-\epsilon) | \,,\, |\log(1+\epsilon)| \} 
    +
    O\Big( \mfrac{\beta^2 \log m}{n} \Big)
    \\
    &\;=\;
    O\Big( 
         \mfrac{\epsilon}{(1-\epsilon) }    
         +
         \mfrac{\beta^2 \log m}{n}
    \Big)
    \;.
\end{align*}
This implies the existence of a sufficiently large universal constant $C > 0$ and a sufficiently small universal constant $c > 0$ such that 
\begin{align*}
    \P\Big( 
        \,
        |\cE - \bar \cE|
        \;\geq\;  
        C
        \Big(
            \mfrac{\epsilon}{(1-\epsilon) }    
            +
            \mfrac{\beta^2 \log m}{n}
        \Big)
    \Big)
    \;\leq\; 
    \epsilon^{-2} 
    e^{ - c \log m }\;.
\end{align*}
\end{proof}

\section{Proof of \Cref{prop:moderate:cold}: Moderate-to-cold temperature regime} \label{sec:proof:moderate:cold}

We seek to prove \Cref{prop:moderate:cold}, the result concerning the moderate-to-cold temperature regime $\beta= \Omega(1)$. To study $\cE$ with the techniques from the Gaussian REM literature, the first step is to approximate each random energy $E_i$ by a Gaussian. A crude approximation follows directly from CLT, since each energy is an empirical average of i.i.d.~random variables whose asymptotic mean is negligible and whose asymptotic variance is one:
\begin{align*}
     E_1 
    \;=\; 
    - \mfrac{\sum_{j=1}^{N-1} Y_{1j} - (N-1) \mu}{  \sigma \, \sqrt{N-1}}
    \;=\;
    - \mfrac{\sqrt{2n} \sum_{j=1}^{N-1} (Y_{1j} + \frac{1}{2n}) }{\sqrt{N-1}}
    \;.
\end{align*}
However, since the REM computation will involve measuring the c.d.f.~of $E_1$ at locations that grow as $\Omega(\sqrt{\log m})$, we require tighter and location-dependent Gaussian approximation error terms compared to the uniform $\frac{1}{\sqrt{n}}$ error from the standard Berry-Ess\'een bound. We obtain such an approximation by applying classical techniques from a Cramér-type moderate deviation theorem \citep{cramer1938nouveau}. In the result below and throughout this section, we write $\Phi$ for the c.d.f.~of $\cN(0,1)$ and $\Phi^c(x) = 1 -\Phi(x)$.

\begin{proposition}[Cramér-type moderate deviation theorem for empirical averages of log-chi-squared variables] \label{prop:nonuniform:CLT:crude} There exists some universal constant $a > 0$ such that 
\begin{align*}
    \P( E_1 > x ) 
    \;=&\;
    \Phi^c(x) 
     \,
     \Big( 1 + O\Big( \mfrac{1}{n} + \mfrac{N^{1/2}}{n^{3/2}}  + \mfrac{1+|x|^3}{\sqrt{N}}\Big) \Big)
    \qquad 
    \text{ for } 
    \quad 
    0 \leq x \leq a N^{1/6}
    \;,
    \\
    \P( E_1 < x ) 
    \;=&\;
    \Phi(x) 
     \,
     \Big( 1 + O\Big( \mfrac{1}{n} + \mfrac{N^{1/2}}{n^{3/2}}  + \mfrac{1+|x|^3}{\sqrt{N}}\Big) \Big)
    \qquad 
    \text{ for } 
    \quad 
    - a N^{1/6} \leq x \leq 0
    \;.
\end{align*}
Suppose in addition that $|x| = o(N^{1/6})$ and $N=o(n^3)$. If $x > 0$ with $x = \omega(1)$, then 
\begin{align*}
    \P( E_1 > x ) 
    \;=\; 
    (1+o(1)) \, e^{ - \frac{x^2}{2} + o(x) }
    \;.
\end{align*}
If $x < 0$ with $|x| = \omega(1)$, then
\begin{align*}
    \P( E_1 < x ) 
    \;=\; 
    (1+o(1)) \, e^{ - \frac{x^2}{2} + o(|x|) }
    \;.
\end{align*}
If instead $x=O(1)$, we have 
\begin{align*}
    \P( E_1 > x ) 
    \;=&\;
    \begin{cases}
        \Theta( e^{-x^2/2} ) 
        & 
        \text{ if } x > 0\;,
        \\
        \frac{1}{2} (1 + o(1))
        & 
        \text{ if } x = o(1)\;,
        \\
        1 - \Theta( e^{-x^2/2} ) 
        & 
        \text{ if } x < 0\;.
    \end{cases}
\end{align*}
\end{proposition}

\subsection{Proof of \Cref{prop:nonuniform:CLT:crude}}

Let $V_1, \ldots, V_N$ be i.i.d.~random variables with zero mean and unit variance such that
\begin{align*}
    \mean[e^{t_0|V_1|}] \leq c
    \qquad 
    \text{ for some universal constants }
    \;\;
    t_0, \; c \,>\, 0
    \;.
    \tagaligneq \label{eq:cramer:condition}
\end{align*}
A classical result due to \cite{cramer1938nouveau} says that
there exist constants $A, a > 0$ that depend only on $t_0$ and $c$ such that
\begin{align*}
    \bigg| 
        \,
        \mfrac{
            \P(\frac{1}{\sqrt{N}} \sum_{i=1}^N V_i > x)
        }{
            \Phi^c(x)
        } 
        \,-\,
        1
        \,
    \bigg|
    \;\leq\;
    \mfrac{A}{\sqrt{N}} (1+x^3)
    \qquad 
    \text{ for } 
    0 \leq x \leq a N^{1/6}
    \;.
    \tagaligneq \label{eq:cramer}
\end{align*}
This is known as the Cram\'er-type moderate deviation theorem. See \cite{petrov1975sums} for a textbook reference, and \cite{liu2023cramer} for recent extensions to the setting of locally dependent variables. Rewriting \eqref{eq:cramer} gives that, for $0 \leq x \leq a N^{1/6}$,
\begin{align*}
    \P\Big( 
        \,\mfrac{1}{\sqrt{N}} \msum_{i=1}^N V_i > x 
    \Big)
    \;=\;
     \Phi^c(x) 
     \,
     \Big( 1 + O\Big( \mfrac{1+|x|^3}{\sqrt{N}}\Big) \Big)
     \;.
     \tagaligneq \label{eq:cramer:pos}
\end{align*}
Moreover, replacing $V_i$ by $-V_i$ in \eqref{eq:cramer} and noting that $\Phi^c(-x) = \Phi(x)$, we get that for $- a N^{1/6} \leq x \leq 0$,
\begin{align*}
    \P\Big( 
        \,\mfrac{1}{\sqrt{N}} \msum_{i=1}^N V_i < x 
    \Big)
    \;=\;
     \Phi(x) 
     \,
     \Big( 1 + O\Big( \mfrac{1+|x|^3}{\sqrt{N}}\Big) \Big)
     \;.
     \tagaligneq \label{eq:cramer:neg}
\end{align*}
We first seek to apply this result to the empirical average 
\begin{align*}
    - \mfrac{1}{\sqrt{N-1}} \msum_{j=1}^{N-1} \Big(\mfrac{Y_{1j} - \mean[Y_{11}]}{\sqrt{\Var[Y_{11}]}} \Big)\;.
\end{align*}
To this end, we first compute 
\begin{align*}
    \mean\Big[ 
        \exp\Big(
                \mfrac{| Y_{11} - \mean[Y_{11}]|}{(\Var[Y_{11}])^{1/2}}
        \Big)
    \Big]
    \;\leq&\;
    \mean\Big[ 
        \exp\Big(
                \mfrac{Y_{11} - \mean[Y_{11}]}{(\Var[Y_{11}])^{1/2}}
        \Big)
    \Big]
    +
    \mean\Big[ 
        \exp\Big(
                - \mfrac{Y_{11} - \mean[Y_{11}]}{(\Var[Y_{11}])^{1/2}}
        \Big)
    \Big]
    \;.
\end{align*}
By \Cref{lem:log:chi:sq},
\begin{align*}
    \mean[Y_{11}]
    \;=\;
    - 
    \mfrac{1}{2n} 
    +
    O\Big( \mfrac{1}{n^2} \Big)
    \;,
    \qquad 
    \Var[Y_{11}]
    \;=\;
    \mfrac{1}{2n}
    + 
    O\Big( \mfrac{1}{n^2} \Big)
    \;.
    \tagaligneq \label{eq:log:chi:sq:mean:var}
\end{align*}
In particular $(\Var[Y_{11}])^{-1/2} = \sqrt{2n} + O(n^{-1/2})$, 
so there exists some universal constant $n_0 \in \N$ such that for all $n \geq n_0$, $(\Var[Y_{11}])^{-1/2} \leq n/2$. Then for all $n \geq n_0$, we can apply \Cref{lem:log:chi:sq} with $\nu = \pm \frac{1}{2} (\Var[Y_{11}])^{-1/2}$ and get that
\begin{align*}
    &\;
    \mean\Big[ 
        \exp\Big(
                \mfrac{Y_{11} - \mean[Y_{11}]}{(\Var[Y_{11}])^{1/2}}
        \Big)
    \Big]
    \\
    &\;=\;
     \exp\Big(
            \mfrac{\frac{1}{4} (\Var[Y_{11}])^{-1} - \frac{1}{2} (\Var[Y_{11}])^{-1/2}}{n}
            -
            \mfrac{\mean[Y_{11}]}{(\Var[Y_{11}])^{1/2}}
            +
            O\Big( 
                \mfrac{(\Var[Y_{11}])^{-3/2}}{n^2} 
            \Big)
        \Big)
    \\
    &\;=\;
     \exp( O(1) )
    \;=\; O(1)
    \;,
\end{align*}
and similarly 
\begin{align*}
    \mean\Big[ 
        \exp\Big(
            -
                \mfrac{Y_{11} - \mean[Y_{11}]}{(\Var[Y_{11}])^{1/2}}
        \Big)
    \Big] \;=\; O(1)
    \;.
\end{align*}
In other words, the condition \eqref{eq:cramer:condition} is satisfied with $t_0 = 1$.  Applying \eqref{eq:cramer:pos} and \eqref{eq:cramer:neg} then gives, for $0 \leq x \leq a N^{1/6}$,
\begin{align*}
    \P\Big( - \mfrac{1}{\sqrt{N-1}} \msum_{j=1}^{N-1} \Big(\mfrac{Y_{1j} - \mean[Y_{11}]}{\sqrt{\Var[Y_{11}]}} \Big) > x \Big) 
    =&\;
    \Phi^c(x) 
     \,
     \Big( 1 + O\Big( \mfrac{1+|x|^3}{\sqrt{N}}\Big) \Big)
    \;,
\end{align*}
and that for $- a N^{1/6} \leq x \leq 0$,
\begin{align*}
    \P\Big( - \mfrac{1}{\sqrt{N-1}} \msum_{j=1}^{N-1} \Big(\mfrac{Y_{1j} - \mean[Y_{11}]}{\sqrt{\Var[Y_{11}]}} \Big) < x \Big)
    =&\;
    \Phi(x) 
     \,
     \Big( 1 + O\Big( \mfrac{1+|x|^3}{\sqrt{N}}\Big) \Big)
    \;.
\end{align*}
To rearrange this into a statement about our target quantity,
\begin{align*}
     E_1 
    \;=\;
    - \mfrac{\sqrt{2n} \sum_{j=1}^{N-1} (Y_{1j} + \frac{1}{2n}) }{\sqrt{N-1}}
    \;,
\end{align*}
we replace $x$ above with 
\begin{align*}
    \tilde x \;\coloneqq\; \mfrac{(2n)^{-1/2} \, x + \sqrt{N-1}\, \mean[Y_{11}] + \frac{\sqrt{N-1}}{2n}}{\sqrt{\Var[Y_{11}]}}
    \;=&\;
    \mfrac{1}{\sqrt{2n \Var[Y_{11}] }}
    \Big(  x + \sqrt{2n(N-1)}  \Big(\mean[Y_{11}] + \mfrac{1}{2n} \Big) \Big)
    \\
    \;\overset{\eqref{eq:log:chi:sq:mean:var}}{=}&\;
    \mfrac{1}{\sqrt{1 + O(n^{-1}) }}
    \Big(  x - O\Big(\mfrac{N^{1/2}}{n^{3/2}}  \Big) \Big)
    \\
    \;=&\;
    x + O\Big( \mfrac{1}{n} + \mfrac{N^{1/2}}{n^{3/2}} \Big)
    \;.
\end{align*}
Since the derivative of $\Phi(x)$ is the standard Gaussian p.d.f., by a first-order Taylor expansion, we obtain that for $0 \leq x \leq a N^{1/6}$,
\begin{align*}
    \P(E_1 > x )
    \;=&\;
    \Phi^c(\tilde x) 
     \,
     \Big( 1 + O\Big( \mfrac{1+|\tilde x|^3}{\sqrt{N}}\Big) \Big)
     \;=\;
     \Phi^c(x) 
     \,
     \Big( 1 + O\Big( \mfrac{1}{n} + \mfrac{N^{1/2}}{n^{3/2}}  + \mfrac{1+| x|^3}{\sqrt{N}}\Big) \Big)
    \;,
\end{align*}
and that for $- a N^{1/6} \leq x \leq 0$,
\begin{align*}
    \P(E_1 < x ) 
    =&\;
    \Phi(x) 
     \,
     \Big( 1 + O\Big( \mfrac{1}{n}+ \mfrac{N^{1/2}}{n^{3/2}}  + \mfrac{1+| x|^3}{\sqrt{N}}\Big) \Big)
    \;.
\end{align*}
This proves the first set of desired bounds. 

\vspace{1em}

Suppose $x > 0$ with $|x| = \omega(1)$. A standard Gaussian tail estimate gives
\begin{align*}
    \Phi^c(x)
    \;=\;
    \mfrac{e^{-x^2/2} }{\sqrt{2\pi}\, x (1+o(1))}
    \;=\;
    e^{ - \frac{x^2}{2} + o(|x|) }
    \;.
\end{align*}
Provided that $|x| = o(N^{1/6})$ and $N = o(n^3)$, we obtain 
\begin{align*}
    \P( E_1 > x ) 
    \;=\; 
    (1+o(1)) \, e^{ - \frac{x^2}{2} + o(|x|) }
    \;.
\end{align*}
Suppose instead $x < 0$ with $|x| = \omega(1)$  and $N = o(n^3)$. Then
\begin{align*}
    \P( E_1 < x ) 
    \;=\; 
    (1+o(1)) \, e^{ - \frac{x^2}{2} + o(|x|) }
    \;.
\end{align*}
If instead $x=O(1)$, we have 
\begin{align*}
    \P( E_1 > x ) 
    \;=&\;
    (1+o(1)) \, \Phi^c(x)
    \;=\;
    \begin{cases}
        \Theta( e^{-x^2/2} ) 
        & 
        \text{ if } x > 0\;,
        \\
        \frac{1}{2} (1 + o(1))
        & 
        \text{ if } x = o(1)\;,
        \\
        1 - \Theta( e^{-x^2/2} ) 
        & 
        \text{ if } x < 0\;.
    \end{cases}
\end{align*}
This gives the second set of desired bounds. \qed

\subsection{Proof body of \Cref{prop:moderate:cold}} 

We seek to study, for $\beta = \Omega(1)$, the quantity
\begin{align*}
    \cE
    \;=&\;
    \log \Big( \mfrac{1}{m} \msum_{i \leq m} e^{ - \beta \sqrt{\log m} \, E_i} \Big)
    +
    2 (N-1)  \mu 
    \;.
\end{align*}
Let $x_0 > 0$ be chosen later, and write
\begin{align*}
    e^{\cE - 2 (N-1) \mu} 
    \;=&\;
    \mfrac{1}{m}
    \msum_{i \leq m}
    \ind_{\{| E_i| \leq x_0 \sqrt{\log m} \}} \, e^{ - \beta \sqrt{\log m} E_i} 
    + 
    \mfrac{1}{m}
    \msum_{i \leq m}
    \ind_{\{| E_i| > x_0 \sqrt{\log m} \}} \, e^{ - \beta \sqrt{\log m} E_i} 
    \\
    \;\eqqcolon&\;
    I_- + I_+
    \;.
\end{align*}
The proof is an adaptation of Varadhan's lemma with an explicit computation of the error bounds, and consists of three steps:
\begin{enumerate}
    \item We show that $I_+$ is negligible with high probability for some sufficiently large $x_0$. This reduces the analysis to a bounded domain;
    \item We use a covering argument on the region $[-x_0, x_0]$ to split $I_-$ further into $(2K+1)$-many segments $(I_j)_{j= -K}^K$. On each segment, we control $I_j$ by its maximum and minimum over the segment, incurring an error that vanishes for a sufficiently large $K$;
    \item We show that provided that $K$ is not too large, only one of the $I_j$'s dominate, which can be computed to give the value of $I_-$ and therefore $\cE$.
\end{enumerate}
The concentration inequality is obtained by a careful choice of $x_0$ and $K$. 

\vspace{.5em}

Our first lemma shows that $I_+$ can be ignored with high probability. Write $A_+ \coloneqq \{ | E_i| \leq  x_0 \sqrt{\log m} \text{ for all } 1 \leq i \leq m\}$, and note that on the event $A_+$, $I_+ = 0$.

\begin{lemma} \label{lem:ignore:I:plus}  $
    1 - \P(A_+)
    \;=\;
    O\Big( m^{1 - \frac{x_0^2}{2} +o\big( \frac{1}{\sqrt{\log m}} \big)  } \Big)
    \;.
$
\end{lemma}

\begin{remark*} If $ x_0 > \sqrt{2}$ with $|x_0 - \sqrt{2}| = \Omega(1)$,  the bound in \Cref{lem:ignore:I:plus} is $o(1)$.
\end{remark*}

\begin{proof}
By a union bound followed by \Cref{prop:nonuniform:CLT:crude}, we have 
\begin{align*}
    1 - \P(A_+)
    \;\leq&\; 
    m\, \P( |E_1| > x_0 \sqrt{\log m}  )
    \\
    \;=&\;
     m\, \P( E_1 > x_0 \sqrt{\log m}  ) 
     +
     m\, \P( E_1 < - x_0 \sqrt{\log m}  )
    \\
    \;=&\;
    O\Big( m^{1 - \frac{x_0^2}{2} +o\big( \frac{1}{\sqrt{\log m}} \big)  } \Big)
    \;.
\end{align*}
\end{proof}

\vspace{.5em}

The next step is to simplify $I_-$ by a covering argument on $[-x_0,x_0]$. Let $K \in \N$ be chosen later, and split $[-x_0, x_0]$ into a disjoint union of $2K+1$ equal-size intervals $(\cB_j)_{j=-K}^K$:
\begin{align*}
    \cB_j 
    \;\coloneqq&\; 
    \Big[ \mfrac{2 j - 1}{2 K+ 1 } \, x_0 \;,\;    \mfrac{2 j + 1}{2 K+ 1 } \, x_0 \Big)
    \qquad \text{ for } -K \leq j \leq K-1
    \\
    \text{ and } \qquad
    \cB_K
    \;\coloneqq&\; 
    \Big[ \mfrac{2 K - 1}{2 K+ 1 } \, x_0 \;,\;    \mfrac{2 K + 1}{2 K+ 1 } \, x_0 \Big]
    \;.
\end{align*}
This allows us to write 
\begin{align*}
    I_- 
    \;=&\;
    \mfrac{1}{m}
    \msum_{i \leq m}
    \ind_{\{| E_i| \leq x_0 \sqrt{\log m} \}} \, e^{ - \beta \sqrt{\log m} E_i} 
    \\
    \;=&\;
    \msum_{j = - K}^K
    \Big(
        \mfrac{1}{m}
        \msum_{i \leq m}
        \ind_{\big\{ - \frac{E_i}{\sqrt{\log m}}   \in \cB_j\big\}} \, e^{ - \beta \sqrt{\log m} \, E_i} 
    \Big)
    \;\eqqcolon\;
    \msum_{j = - K}^K \, I_j
    \;.
\end{align*}
For $\cB \subseteq \R$, denote the empirical measure
\begin{align*}
    F_m(\cB) \;\coloneqq\; \mfrac{1}{m} \msum_{i \leq m} \ind_{\big\{ - \frac{E_i}{\sqrt{\log m}} \in \cB \big\}}   \;,
\end{align*}
and also denote the values of $x \beta \log m$ at the two endpoints of $I_j$ as
\begin{align*}
    y_{j,-} \;\coloneqq&\; \mfrac{2j-1}{2K+1} x_0  \beta \log m
    &\text{ and }&&
    y_{j,+} \;\coloneqq&\; \mfrac{2j+1}{2K+1} x_0 \beta \log m\;.
\end{align*}
By considering the maximum and minimum summand in $I_j$, we can control
\begin{align*}
    I_j 
    \;\leq&\;  
    F_m(\cB_j) 
    \,
    \max_{1 \leq i \leq m}  \,
    \Big\{ 
        \ind_{\big\{ - \frac{E_i}{\sqrt{\log m}} \in \cB_j \big\}}  
        e^{ - \beta \sqrt{\log m} E_i}  
    \Big\}
    \\
    \;\leq&\;
    F_m(\cB_j) 
    \,
    \sup_{x \in \cB_j} \,e^{ x \beta \log m}
    \;=\;
    F_m(\cB_j) 
    \,
    e^{ y_{j,+}}
    \;,
    \\
    I_j 
    \;\geq&\;  
    F_m(\cB_j) 
    \,
    \min_{1 \leq i \leq m}  \,
    \Big\{ 
        \ind_{\big\{ - \frac{E_i}{\sqrt{\log m}} \in \cB_j \big\}}  
        e^{ - \beta \sqrt{\log m} E_i}  
    \Big\}
    \\
    \;\geq&\;
    F_m(\cB_j) 
    \,
    \inf_{x \in \cB_j} \,e^{ x \beta \log m}
    \;=\;
    F_m(\cB_j) 
    \,
    e^{ y_{j,-}}
    \;.
\end{align*}
This implies 
\begin{align*}
    \msum_{j=-K}^K  F_m(\cB_j) \, e^{ y_{j,-}}
    \;\leq\; 
    I_- 
    \;\leq\; 
    \msum_{j=-K}^K  F_m(\cB_j) \, e^{ y_{j,+}}
    \;.
    \tagaligneq \label{eq:I:minus:tmp:control}
\end{align*}
To simplify \eqref{eq:I:minus:tmp:control}, we notice that every $F_m(\cB_j)$ is also an empirical average of $m$ i.i.d.~Bernoulli random variables, each with parameter 
\begin{align*}
    P(\cB_j) \;\coloneqq\; \P\Big( - \mfrac{E_1}{\sqrt{\log m}} \in \cB_j \Big)    \;.
\end{align*} 
This allows us to simplify \eqref{eq:I:minus:tmp:control} by exploiting the concentration of $F_m(\cB_j)$. In the next lemma, for $\epsilon > 0$, we denote the event
\begin{align*}
    A_\epsilon \;\coloneqq\; \big\{ | F_m(\cB_j)  - P(\cB_j) | \leq \epsilon \, P(\cB_j) \quad \text{ for all } \;\; -K \leq j \leq K   \big\}\;.
\end{align*}

\begin{lemma} \label{lem:conc:I:minus} Suppose $x_0 = \Theta(1)$ and $K = o(\sqrt{\log m})$. Then for any $\epsilon > 0$,
\begin{align*}
    1 - \P(A_\epsilon)
    \;=\; 
     O\Big( \epsilon^{-2} \, m^{ - 1 +  \frac{x_0^2}{2} \frac{(2K-1)^2}{(2K+1)^2}  + o(\frac{1}{\sqrt{\log m}}) }   \Big)
    \;.
\end{align*}
\end{lemma}

\begin{remark*} Note that if $x_0 > \sqrt{2}$ with $|x_0 - \sqrt{2}| = \Omega(1)$ and $K = \omega(1)$, the bound in \Cref{lem:conc:I:minus} is $o(1)$ for every fixed $\epsilon > 0$.
\end{remark*}

\begin{proof} By a union bound followed by Markov's inequality, we have that
\begin{align*}
    1 - \P(A_\epsilon)
    \;\leq&\;
    \msum_{j= - K}^K \P\big(  | F_m(\cB_j)  - P(\cB_j) | > \epsilon \, P(\cB_j)   \big)
    \\
    \;\leq&\;
    \msum_{j= - K}^K \mfrac{\Var[ F_m(\cB_j) ]}{\epsilon^2 P(\cB_j)^2}
    \\
    \;=&\;
    \msum_{j= - K}^K \mfrac{P(\cB_j)\, (1- P(\cB_j))}{\epsilon^2 m P(\cB_j)^2}
    \;\leq\;
    \msum_{j= - K}^K \mfrac{1}{\epsilon^2 m P(\cB_j)}
    \;.
    \tagaligneq \label{eq:A:eps:init}
\end{align*}
Recalling the definition of $P(\cB_j)$, we can express
\begin{align*}
    m P(\cB_j)
    \;=&\;
    m \, \P\Big( \mfrac{y_{j,-}}{\beta \log m} \leq  - \mfrac{E_1}{\sqrt{\log m}} \leq \mfrac{y_{j,+}}{\beta \log m} \Big)
    \\
    \;=&\;
    m \, \P\Big(  - E_1 \geq \mfrac{y_{j,-}}{\beta \sqrt{\log m}} \Big)
    -
    m \, \P\Big(  - E_1 > \mfrac{y_{j,+}}{\beta \sqrt{\log m}} \Big)
    \\
    \;=&\;
    -
    m \, \P\Big( - E_1 < \mfrac{y_{j,-}}{\beta \sqrt{\log m}} \Big)
    +
    m \, \P\Big(  - E_1 \leq \mfrac{y_{j,+}}{\beta \sqrt{\log m}} \Big)
    \;.
\end{align*}
Using $x_0 = \Theta(1)$ and recalling the assumption that $\log m = o(N^{1/3})$, we get that 
\begin{align*}
    \max\Big\{ 
        \Big| \mfrac{y_{j,-}}{\beta \sqrt{\log m}}\Big|
        \,,\,
        \Big| \mfrac{y_{j,+}}{\beta \sqrt{\log m}}\Big|
    \Big\} 
    \;=\;
    O( x_0 \sqrt{\log m} ) 
    \;=\;
    o(N^{1/6})\;.
\end{align*}
This allows us to apply \Cref{prop:nonuniform:CLT:crude} to the expression above. Noting that $E_1$ is a continuous random variable, we obtain
\begin{align*}
    m P(\cB_j) \;=\;  m (\tilde P( y_{j,-} ) - \tilde P( y_{j,+} ) )
    \;,
    \tagaligneq \label{eq:PBj:tilde:P}
\end{align*}
where we have denoted
\begin{align*}
    \tilde P(y)
    \;\coloneqq&\;
    \begin{cases}
    (1+o(1))
    \,
    \exp\big( - \frac{y^2}{2 \beta^2 \log m} + o\big( \frac{y}{\beta \sqrt{\log m}} \big) \big)
    & \text{ if } y > 0 \text{ and } y = \omega(\beta \sqrt{\log m})\;,
    \\
    (1+o(1)) \, \Phi\big( - \frac{y}{\beta \sqrt{\log m}} \big)
    & \text{ if } y = O(\beta \sqrt{\log m})\;,
    \\
    1 - 
    (1+o(1))
    \,
    \exp\big( - \frac{y^2}{2 \beta^2 \log m} + o\big( \frac{y}{\beta \sqrt{\log m}} \big) \big)
    & \text{ if } y < 0 \text{ and } y = \omega(\beta\sqrt{\log m})\;,
    \end{cases}
    \\
    \;=&\;
    \begin{cases}
        \Theta\big( 
            \,\exp\big( - \frac{y^2}{2 \beta^2 \log m} + o\big( \frac{y}{\beta \sqrt{\log m}} \big) \big)\,    
        \big)
        &\text{ if } y > 0\;,
        \\
        1 - \Theta\big( 
            \,\exp\big( - \frac{y^2}{2 \beta^2 \log m} + o\big( \frac{y}{\beta \sqrt{\log m}} \big) \big)\,    
        \big)
        &\text{ if } y < 0\;.
    \end{cases}
\end{align*}
This allows us to compute $m P(\cB_j)$:
\begin{proplist}
    \item For $j \geq 1$, we have $y_{j,+} > y_{j,-} > 0$ and $\tilde P(y_{j,+}) \ll \tilde P(y_{j,-})$, which implies
    \begin{align*}
        m \,  P(\cB_j)
        \;=\;
        \Theta\Big(
            m \exp\Big( - \mfrac{y_{j,-}^2}{2 \beta^2 \log m} + o\Big( \mfrac{y_{j,-}}{\beta \sqrt{\log m}} \Big) \Big)
        \Big)
        \;=\;
        \Theta\bigg(
        m^{ 
            1 - \frac{(2j-1)^2 x_0^2}{2(2K+1)^2} 
            + o \big( \frac{1}{\sqrt{\log m}} \big)
        }
        \bigg)
        \;,
    \end{align*}
    where we have used $y_{j,-} = O(\beta\log m)$ and $\beta=\Omega(1)$ in the second equality;
    \item For $j \leq - 1$, we have $0 > y_{j,+} > y_{j,-}$ and $1 - \tilde P(y_{j,-}) \ll 1 -  \tilde P(y_{j,+})$, which implies
    \begin{align*}
        m \,  P(\cB_j)
        \;=\;
        \Theta\Big(
            m \exp\Big( - \mfrac{y_{j,+}^2}{2 \beta^2 \log m} + o\Big( \mfrac{y_{j,+}}{\beta \sqrt{\log m}} \Big) \Big)
        \Big)
        \;=\;
        \Theta\bigg(
        m^{ 
            1 - \frac{(2j+1)^2 x_0^2}{2(2K+1)^2} 
            + o \big( \frac{1}{\sqrt{\log m}} \big)
        }
        \bigg)
        \;;
    \end{align*}
    \item For $j = 0$,  we have $y_{0,-} < 0 <  y_{0,+}$,  $\tilde P(y_{0,+}) \ll \tilde P(y_{0,-})$, which implies
    \begin{align*}
        m \,  P(\cB_0)
        \;=\;
        m 
        +
        \Theta\Big(
            m \exp\Big( - \mfrac{y_{0,-}^2}{2 \beta^2 \log m} + o\Big( \mfrac{y_{0,-}}{\beta \sqrt{\log m}} \Big) \Big)
        \Big)
        \;=\;
        m(1+o(1))
        \;,
    \end{align*}
    where, in the second equality, we have used $K= o(\sqrt{\log m})$ and $x_0 = \Omega(1)$ to obtain that 
    \begin{align*}
        \mfrac{y_{0,-}^2}{2 \beta^2 \log m} 
        \;=\;
        \mfrac{x_0^2}{(2K+1)^2} \log m \;=\; \omega(1)\;.
    \end{align*}
\end{proplist} 
Substituting these calculations into \eqref{eq:A:eps:init}, we obtain
\begin{align*}
    1 - \P(A_\epsilon)
    &\;=\;
     O\bigg(
        \mfrac{1}{\epsilon^2 m^{1 + o(\frac{1}{\sqrt{\log m}})}}
        \,
        \bigg(
            \sum_{j=-K}^{-1}
            m^{ 
                \frac{(2j+1)^2 x_0^2}{2(2K+1)^2} 
            }
            +
            1
            +
            \sum_{j=1}^{K}
            m^{ 
                \frac{(2j-1)^2 x_0^2}{2(2K+1)^2} 
            }
        \bigg)
    \bigg) 
    \\
    &\;=\;
    O\bigg(
        \mfrac{1}{\epsilon^2 m^{1 + o(\frac{1}{\sqrt{\log m}})}}
        \,
        m^{\frac{x_0^2}{2} \frac{(2K-1)^2}{(2K+1)^2}  }
        \,
            \sum_{j=1}^{K}
            m^{ 
                \frac{- (2K-1)^2 + (2j-1)^2 }{2(2K+1)^2} 
                x_0^2
            }
    \bigg) 
    \\
    &\;\overset{(a)}{=}\;
    O\bigg(
        \mfrac{1}{\epsilon^2 m^{1 + o(\frac{1}{\sqrt{\log m}})}}
        \,
        m^{\frac{x_0^2}{2} \frac{(2K-1)^2}{(2K+1)^2} }
        \,
            \sum_{j=1}^{K}
            m^{ 
                -
                \frac{(2K-2j)^2 }{2(2K+1)^2} 
                x_0^2
            }
    \bigg) 
    \\
    &\;\overset{(b)}{=}\;
    O\bigg(
        \mfrac{1}{\epsilon^2 m^{1 + o(\frac{1}{\sqrt{\log m}})}}
        \,
        m^{\frac{x_0^2}{2} \frac{(2K-1)^2}{(2K+1)^2} }
        \,
            \bigg(
                1
                +
                \sum_{j=1}^{K-1}
                m^{ 
                    -
                    \frac{2 j^2 }{(2K+1)^2} 
                    x_0^2
                }
            \bigg)
    \bigg) 
    \\
    &\;=\;
    O\bigg(
        \mfrac{1}{\epsilon^2 m^{1 + o(\frac{1}{\sqrt{\log m}})}}
        \,
        m^{\frac{x_0^2}{2} \frac{(2K-1)^2}{(2K+1)^2} }
        \,
        \bigg(
            1
            +
            \mint_0^{K-1}
            e^{ 
                -
                \frac{2 u^2 }{(2K+1)^2} 
                x_0^2
                (\log m)
            }
            du
        \bigg)
    \bigg) 
    \\
    &\;\overset{(c)}{=}\;
    O\bigg(
        \mfrac{1}{\epsilon^2 m^{1 + o(\frac{1}{\sqrt{\log m}})}}
        \,
        m^{\frac{x_0^2}{2} \frac{(2K-1)^2}{(2K+1)^2}  }
        \,
        \Big(
            1
            +
            \mfrac{2K+1}{2x_0 \sqrt{\log m} }
        \Big)
    \bigg) 
    \\
    &\;\overset{(d)}{=}\; O\Big( \epsilon^{-2} \, m^{ - 1 +  \frac{x_0^2}{2} \frac{(2K-1)^2}{(2K+1)^2}  + o(\frac{1}{\sqrt{\log m}}) }   \Big)
    \;.
\end{align*}
In $(a)$, we have used that $-(2K-1)^2 + (2j-1)^2 = - (2K-2j)(2K+2j-2) \leq -(2K-2j)^2$; in $(b)$ we have used a change of index; in $(c)$ we have used a change-of-variable to bound the Gaussian integral; in $(d)$, we have used that $x_0 = \Theta(1)$ and that, since $K = o(\sqrt{\log m})$, 
\begin{align*}
    \mfrac{2K+1}{2x_0 \sqrt{\log m}}
    \;=\; 
    o(1)
    \;.
\end{align*}
\end{proof}

\vspace{1em}

Conditioning on $A_\epsilon$, we can now replace \eqref{eq:I:minus:tmp:control} by the bound
\begin{align*}
    (1-\epsilon)  I_-^\# \;\leq\; I_- \;\leq\; (1+\epsilon)  I_-^*\;,
\end{align*}
where 
\begin{align*}
        I_-^\# \;\coloneqq&\; \msum_{j=-K}^K   P(\cB_j) e^{ y_{j,-}}
    &\text{ and }&&
        I_-^* \;\coloneqq&\; \msum_{j=-K}^K  P(\cB_j) e^{ y_{j,+}}
    \;.
\end{align*}
To compute $I_-^\#$ and $I^*_-$, we recall that $\max\{ |y_{j,-}| , |y_{j,+}|\} = O(\beta \log m)$, and use \eqref{eq:PBj:tilde:P} with the more explicit formula for $\tilde P$ to obtain that 
\begin{align*}
    &\;P(\cB_j) 
    \;=\;
    \tilde P( y_{j,-} ) - \tilde P( y_{j,+} ) 
    \\
    &\;=\;
    \begin{cases}
        (1+o(1))
        \,
        \exp\big( - \frac{y_{j,-}^2}{2 \beta^2 \log m} + 
        o(\sqrt{\log m})
        \big)
        & \text{ if } j \geq 1 \;\;\text{ and }\;\; y_{j,-} = \omega(\beta \sqrt{\log m}) \;,
        \\[1em]
        (1+o(1)) \, \Phi\big( - \frac{y_{j,-}}{\beta \sqrt{\log m}} \big)
        & \text{ if } j \geq 1 \;\;\text{ and }\;\; y_{j,-} = O(\beta \sqrt{\log m}) \;,
        \\[1em]
        1 + o(1)
        &\text{ if } j = 0\;,
        \\[1em]
        (1+o(1))
        \;
        \Phi\big( - \frac{y_{j,+}}{\beta \sqrt{\log m}} \big)
        &\text{ if } j \leq - 1  \;\;\text{ and }\;\; y_{j,+} = O(\beta \sqrt{\log m}) \;, 
        \\[1em]
        (1+o(1))
        \;
        \exp\big( - \frac{y_{j,+}^2}{2 \beta^2 \log m} + o(\sqrt{\log m}) \big)
        &\text{ if } j \leq - 1  \;\;\text{ and }\;\; y_{j,+} = \omega(\beta \sqrt{\log m}) \;, 
    \end{cases}
    \\[1em]
    &\;=\;
    \begin{cases}
        \Theta\Big(  \exp\big( - \frac{y_{j,-}^2}{2 \beta^2 \log m} + o(\sqrt{\log m}) \big) \Big)
        & \text{ if } j \geq 1 \;,
        \\[1em]
        1 + o(1)
        &\text{ if } j = 0\;,
        \\[1em]
        \Theta\Big(  \exp\big( - \frac{y_{j,+}^2}{2 \beta^2 \log m} +o(\sqrt{\log m}) \big) \Big)
        &\text{ if } j \leq - 1  \;.
    \end{cases}
\end{align*}
The idea is that in both $I^\#_-$ and $I^*_-$, only one summand dominates: Informally, writing $y_j = \frac{y_{j,+} + y_{j,-}}{2}= \frac{2j}{2K+1} x_0 \beta \log m$, we have 
\begin{align*}
    P(\cB_j) \, e^{y_j} 
    \;\approx\; 
    \exp\Big( - \mfrac{y_j^2}{2 \beta^2 \log m} + y_j \Big)
    \;=\;
    \exp\Big( 
        - 
        \mfrac{(y_j - \beta^2 \log m )^2}{2 \beta^2 \log m} 
        +
        \mfrac{\beta^2 \log m}{2}
    \Big)
    \;,
\end{align*}
so if $\beta \in (0, x_0]$, the dominant term in both of the sums $I_-^\#$ and $I^*_-$ is the $j$-th term such that $y_j \approx \beta^2 \log m$. If $\beta > x_0$, the dominant term is the one with $j= K$. These two cases correspond to the two phases of $\bar \cE$ and $x_0$ will turn out to be the threshold at which phase transition happens.

\vspace{.5em}

To make this precise, observe that by definition 
\begin{align*}
    I^*_-
    \;=\;
    I^\#_-
    \,
    e^{ \frac{2 x_0 \beta \log m}{2K+1}  }
    \;,
    \tagaligneq \label{eq:relate:I:star:I:sharp}
\end{align*}
so it suffices to compute $I^\#_-$. We first identify the dominant term in $I^\#_-$ by defining the index 
\begin{align*}
    j^\# 
    \;\coloneqq\;
    \argmin_{-K \leq j \leq K}
    \big| y_{j,-} - \beta^2 \log m \big| 
    \;=\;
    \argmin_{-K \leq j \leq K}
    \Big| \mfrac{2j-1}{2K+1} x_0 - \beta \Big| 
    \;,
\end{align*} 
where the smaller index is taken as the $\argmin$ in the case of a tie, and denote 
\begin{align*}
    \delta_\# \;\coloneqq\;
    \mfrac{| y_{j^\#,-} - \beta^2 \log m |}{\beta \sqrt{\log m}} 
    \;.
\end{align*}
We seek to control the error
\begin{align*}
    \Delta_\#
    \;\coloneqq\; 
    \mfrac{\big| I_-^\# - P(\cB_{j^\#}) \, e^{y_{j^\#,-}} \big|}{P(\cB_{j^\#})  \, e^{y_{j^\#,-}}}
    \;\leq\;
    \msum_{j \neq j^\#}
    \mfrac{
        P(\cB_j)
        \,
        e^{y_{j,-}}
     }{P(\cB_{j^\#})  \, e^{y_{j^\#,-}}}
     \;.
\end{align*}
The next lemma provides a control on the individual terms.

\begin{lemma} \label{lem:compute:P:e} Suppose $x_0 = \Theta(1)$. The following bounds hold:
\begin{align*}
    P(\cB_{j^\#})  \, e^{y_{j^\#,-}}
    \;=&\;
    (1+o(1))
    \,
    \exp\Big( - \mfrac{\delta_\#^2}{2}   
    +
    \mfrac{\beta^2 \log m}{2} 
     + o(\sqrt{\log m}) \Big)
    \;,
    \tagaligneq \label{eq:calc:I:sharp:main}
    \\
    \mfrac{
        P(\cB_0)
        \,
        e^{y_{0,-}}
     }{P(\cB_{j^\#})  \, e^{y_{j^\#,-}}}
     \;=&\;
     O\Big( \exp\Big(  \mfrac{\delta_\#^2}{2}   
    -
    \mfrac{\beta^2 \log m}{2} 
     + o(\sqrt{\log m}) \Big) \Big)
     \;,
     \\
    \mfrac{
        P(\cB_j)
        \,
        e^{y_{j,-}}
     }{P(\cB_{j^\#})  \, e^{y_{j^\#,-}}}
   \;=&\;
    O\Big(
        \exp\Big( 
            - \mfrac{y_{j,+}^2}{2 \beta^2 \log m} 
            -
            \mfrac{\beta^2 \log m}{2} 
            + 
            \mfrac{\delta_\#^2}{2}   
            +
            o(\sqrt{\log m})
        \Big) 
    \,\Big)
    \quad
    \text{ for } j \leq -1 \;,
    \\
    \mfrac{
        P(\cB_j)
        \,
        e^{y_{j,-}}
     }{P(\cB_{j^\#})  \, e^{y_{j^\#,-}}}
    \;=&\;
     \Theta\Big( 
        \exp\Big( 
            - \mfrac{(y_{j,-} - \beta^2 \log m)^2}{2 \beta^2 \log m} 
            +
            \mfrac{\delta_\#^2}{2}    
            +
            o(\sqrt{\log m})
        \Big) 
    \Big)
    \quad
    \text{ for } j \geq 1 \text{ and } j \neq j^\# \;.
\end{align*}
\end{lemma}

\begin{proof} Observe that we either have $\beta \in (0,x_0]$ and $y_{j^\#,-}$ is close to $\beta^2 \log m = \Theta(x_0^2 \log m) = \Theta(\log m)$, or we have $\beta > x_0$ and $y_{j^\#, -} = y_{K,-} = \frac{2K-1}{2K+1} x_0 \beta \log m = \Theta(\beta \log m)$. In either case, $y_{j^\#,-}$ is positive, $\omega(1)$ and $O(\beta \log m)$. Plugging in the formula for $P(\cB_j)$ and using a completion-of-squares gives the formula for $P(\cB_{j^\#})  \, e^{y_{j^\#,-}}$. 
\begin{align*}
    P(\cB_{j^\#})  \, e^{y_{j^\#,-}}
    \;=&\;
    (1+o(1))
        \,
        \exp\Big( - \mfrac{y_{j^\#,-}^2}{2 \beta^2 \log m} + y_{j^\#,-} + o(\sqrt{\log m}) \Big)
    \\
    \;=&\;
    (1+o(1))
    \,
    \exp\Big( - \mfrac{\delta_\#^2}{2}   
    +
    \mfrac{\beta^2 \log m}{2} 
     + o(\sqrt{\log m}) \Big)
    \;.
\end{align*}
The bound for $j = 0$ follows immediately since $P(\cB_0) = (1+o(1))$ and $y_{0,-} = \frac{-1}{2K+1} x_0 \beta \log m < 0$. For $j \leq -1$, we note that $y_{j,-} < 0$ to compute
\begin{align*}
    \mfrac{
        P(\cB_j)
        \,
        e^{y_{j,-}}
     }{P(\cB_{j^\#})  \, e^{y_{j^\#,-}}}
     \;=&\;
    \Theta\Big(
        \exp\Big( 
            - \mfrac{y_{j,+}^2}{2 \beta^2 \log m} + y_{j,-}
            + 
            \mfrac{\delta_\#^2}{2}   
            -
            \mfrac{\beta^2 \log m}{2} 
            +
            o(\sqrt{\log m})
        \Big) 
    \,\Big)
    \\
   \;=&\;
    O\Big(
        \exp\Big( 
            - \mfrac{y_{j,+}^2}{2 \beta^2 \log m} 
            -
            \mfrac{\beta^2 \log m}{2} 
            + 
            \mfrac{\delta_\#^2}{2}   
            +
            o(\sqrt{\log m})
        \Big) 
    \,\Big)
    \;.
\end{align*}
For $j \geq 1$ with $j \neq j^\#$, we use a completion-of-squares to obtain
\begin{align*}
    \mfrac{
        P(\cB_j)
        \,
        e^{y_{j,-}}
     }{P(\cB_{j^\#})  \, e^{y_{j^\#,-}}}
     \;=&\;
    \Theta\Big(
        \exp\Big( 
            - \mfrac{y_{j,-}^2}{2 \beta^2 \log m} + y_{j,-}
            + 
            \mfrac{\delta_\#^2}{2}   
            -
            \mfrac{\beta^2 \log m}{2} 
            +
            o(\sqrt{\log m})
        \Big) 
    \,\Big)
    \\
    \;=&\;
     \Theta\Big( 
        \exp\Big( 
            - \mfrac{(y_{j,-} - \beta^2 \log m)^2}{2 \beta^2 \log m} 
            +
            \mfrac{\delta_\#^2}{2}    
            +
            o(\sqrt{\log m})
        \Big) 
    \Big)
    \;.
\end{align*}
\end{proof}

\vspace{.5em}

We are ready to use the control on $\Delta_\#$ to compute $I_-$ in the two cases where $\beta \in (0, x_0]$ and $\beta > x_0$. The first lemma focuses on $\beta \in (0,x_0]$.

\begin{lemma} \label{lem:I:minus:beta:smaller} Assume $x_0 = \Theta(1)$ and $K=\Theta((\log m)^{1/4})$. For $\beta \in (0, x_0]$, we have that conditioning on $A_\epsilon$ with $\epsilon \in (0,1)$,
\begin{align*}
    \Big| \mfrac{2 \log I_-}{\beta^2 \log m}   - 1\Big|
    \;=&\;
    O\Big( 
        \mfrac{\epsilon}{(1-\epsilon) \log m}
        +
        \mfrac{1}{(\log m)^{1/4}}
    \Big)
    \;.
\end{align*}
\end{lemma}

\begin{proof} In the case $\beta \in (0, x_0]$, $y_{j^\#,-}$ is close to $\beta^2 \log m$. By noting that $y_{j,+} - y_{j,-} = \frac{2x_0 \beta \log m}{2K+1} $ for all $j$, we can control the approximation error as
\begin{align*}
    \delta_\# 
    \;=\;
    \mfrac{| y_{j^\#,-} - \beta^2 \log m |}{\beta \sqrt{\log m}} 
    \;\leq\; 
    \mfrac{2 x_0 \sqrt{\log m}}{2K+1} 
    \;=\; 
    O\Big(\mfrac{\sqrt{\log m}}{K} \Big) 
    \;=\; 
    O( (\log m)^{1/4} ) \;,
    \tagaligneq \label{eq:delta:sharp:bound:one}
\end{align*}
where, in the last inequality, we have used
$K=\Theta((\log m)^{1/4})$. \color{black} Summing the calculations of $P(\cB_j) e^{y_{j,-}} / P(\cB_{j^\#})  \, e^{y_{j^\#,-}}$ in \Cref{lem:compute:P:e}, we obtain
\begin{align*}
     \Delta_\#
    =
    O\bigg(
    e^{o( \sqrt{\log m}) }
    \bigg(
    \underbrace{ 
        \sum_{j=-K}^{-1}
    e^{
        - \frac{y_{j,+}^2}{2\beta^2 \log m} - \frac{\beta^2 \log m}{2}
    }
    }_{\eqqcolon \, \Delta_\#^- }
    +
    O\Big( e^{ - \frac{\beta^2 \log m}{2} }\Big)
    +
    \underbrace{ 
        \sum_{j=1}^K 
    \ind_{\{j \neq j^\#\}}
    e^{
        - \frac{ (y_{j,-} - \beta^2 \log m)^2 }{2 \beta^2 \log m}
    }
    }_{\eqqcolon \, \Delta_\#^+ }
    \bigg)
    \bigg)
    \;.
\end{align*}
We first control $\Delta^+_\#$. Notice that $(y_{j,-} - \beta^2 
\log m)^2$ is decreasing in $j$ for $j < j^\#$ and increasing in $j$ for $j > j^\#$. This allows us to control 
\begin{align*}
    &\Delta_\#^+
    \;=\;
    \ind_{\{ j^\# > 1\}}
    \exp\Big(
        - \mfrac{ (y_{j^\# - 1,-} - \beta^2 \log m)^2 }{2 \beta^2 \log m}
    \Big)
    +
    \ind_{\{ j^\# > 2 \}}
    \sum_{j=1}^{j^\#-2}
    \exp\Big(
        - \mfrac{ (y_{j,-} - \beta^2 \log m)^2 }{2 \beta^2 \log m}
    \Big)
    \\
    &\;\qquad
    +
    \ind_{\{ j^\# < K\}}
    \exp\Big(
        - \mfrac{ (y_{j^\# + 1,-} - \beta^2 \log m)^2 }{2 \beta^2 \log m}
    \Big)
    +
    \ind_{\{ j^\# < K-1\}}
    \sum_{j=j^\#+2}^K
    \exp\Big(
        - \mfrac{ (y_{j,-} - \beta^2 \log m)^2 }{2 \beta^2 \log m}
    \Big)
    \\
    &\;\leq\;
    \ind_{\{ j^\# > 1\}}
    \exp\Big(
        - \mfrac{ (y_{j^\# - 1,-} - \beta^2 \log m)^2 }{2 \beta^2 \log m}
    \Big)
    +
    \ind_{\{ j^\# > 2 \}}
    \int_1^{j^\#-1}
        \exp\Big(
            - \mfrac{ (y_{u,-} - \beta^2 \log m)^2 }{2 \beta^2 \log m}
        \Big)
    du
    \\
    &\;\quad
    +
    \ind_{\{ j^\# < K\}}
    \exp\Big(
        - \mfrac{ (y_{j^\# + 1,-} - \beta^2 \log m)^2 }{2 \beta^2 \log m}
    \Big)
    +
    \ind_{\{ j^\# < K-1\}}
    \int_{j^\#+1}^K
        \exp\Big(
            - \mfrac{ (y_{u,-} - \beta^2 \log m)^2 }{2 \beta^2 \log m}
        \Big)
    du
    \;.
\end{align*}
To control the terms above, we first recall that by the definition of $j^\#$,
\begin{align*}
    \mfrac{\beta^2 \log m - y_{j^\# - 1, -}}{\beta \sqrt{\log m}}
    \;\geq\; 
    \mfrac{| y_{j^\#,-} - \beta^2 \log m |}{\beta \sqrt{\log m}} 
    \qquad 
    \text{ and }
    \qquad 
    \mfrac{y_{j^\# + 1, -} - \beta^2 \log m }{\beta \sqrt{\log m}}
    \;\geq\;
    \mfrac{| y_{j^\#,-} - \beta^2 \log m |}{\beta \sqrt{\log m}} 
    \;,
\end{align*}
and therefore by the triangle inequality,
\begin{align*}
    \mfrac{\beta^2 \log m - y_{j^\# - 1, -}}{\beta \sqrt{\log m}}
    \;\geq&\; 
    \mfrac{\beta^2 \log m - y_{j^\# - 1, -}}{2 \beta \sqrt{\log m}}
    +
    \mfrac{| y_{j^\#,-} - \beta^2 \log m |}{2\beta \sqrt{\log m}} 
    \;\geq\;
    \mfrac{y_{j^\#,-} - y_{j^\# - 1, -}}{2\beta \sqrt{\log m}} 
    \;=\;
    \mfrac{x_0 \sqrt{\log m}}{2K+1}
    \;,
    \\
    \mfrac{y_{j^\# + 1, -} - \beta^2 \log m }{\beta \sqrt{\log m}}
    \;\geq&\;
    \mfrac{y_{j^\# + 1, -} - \beta^2 \log m }{2\beta \sqrt{\log m}}
    +
    \mfrac{| y_{j^\#,-} - \beta^2 \log m |}{2\beta \sqrt{\log m}} 
    \;\geq\;
    \mfrac{y_{j^\# + 1, -} - y_{j^\#,-}  }{2\beta \sqrt{\log m}}
    \;=\;
    \mfrac{x_0 \sqrt{\log m}}{2K+1}
    \;.
\end{align*} 
This implies 
\begin{align*}
    \max\Big\{ 
        \exp\Big(
        - \mfrac{ (y_{j^\# - 1,-} - \beta^2 \log m)^2 }{2 \beta^2 \log m}
        \Big) 
        \,,\,
        \exp\Big(
            - \mfrac{ (y_{j^\# + 1,-} - \beta^2 \log m)^2 }{2 \beta^2 \log m}
        \Big)
    \Big\} 
    \;=&\;
    O\Big( e^{ - \frac{x_0^2 \log m}{2(2K+1)^2}   }  \Big)
    \\
    \;=&\;
    O\Big( m^{ - \frac{x_0^2 }{2(2K+1)^2} }  \Big)
    \;.
\end{align*}
By recalling that $y_{j,-} = \frac{2j-1}{2K+1} x_0 \beta \log m$, we can compute
\begin{align*}
    &\;
    \int_1^{j^\#-1}
        \exp\Big(
            - \mfrac{ (y_{u,-} - \beta^2 \log m)^2 }{2 \beta^2 \log m}
        \Big)
    du
    \;=\;
    \int_1^{j^\#-1}
    e^{
        - \frac{1}{2} \big(\frac{2 u x_0 \sqrt{\log m}}{2K+1}  - \frac{x_0\sqrt{\log m}}{2K+1} -  \beta \sqrt{\log m} \big)^2
    }
    du
    \\
    &\;=\;
    \mfrac{2K+1}{2x_0 \sqrt{\log m}}
     \int_{y_{1, -} / (\beta \sqrt{\log m})}^{y_{j^\# - 1, -} / (\beta \sqrt{\log m})}
    e^{
        - \frac{1}{2} (u -  \beta \sqrt{\log m} )^2
    }
    du
    \\
    &\;\leq\;
    \mfrac{2K+1}{2x_0 \sqrt{\log m}} 
    \,
    \Phi\Big( - \Big( \beta \sqrt{\log m} - \mfrac{y_{j^\# - 1, -}}{\beta \sqrt{\log m}} \Big) \Big)
    \\
    &\;\leq\;
    \mfrac{2K+1}{2x_0 \sqrt{\log m}} 
    \,
    \Phi\Big( - \mfrac{x_0 \sqrt{\log m}}{2K+1} \Big)
    \\
    &\;=\;
    O\Big( \mfrac{K^2}{\log m } \, e^{ - \frac{x_0^2 \log m}{2(2K+1)^2}   }  \Big)
    \;=\;
    O\Big( m^{ - \frac{x_0^2 }{2(2K+1)^2} }  \Big)
\end{align*}
where we have applied a standard estimation of the Gaussian c.d.f.~and $K=\Theta((\log m)^{1/4})=O( \sqrt{\log m})$  in the last line. By the same argument, we obtain 
\begin{align*}
    \mint_{j^\#+1}^K
    \exp\Big(
            - \mfrac{ (y_{u,-} - \beta^2 \log m)^2 }{2 \beta^2 \log m}
        \Big)
    du
    \;=\;
    O\Big( m^{ - \frac{x_0^2 }{2(2K+1)^2} }  \Big)
    \;.
\end{align*}
Combining the bounds, we obtain 
\begin{align*}
    \Delta_\#^+ 
    \;=\; O\Big( m^{ - \frac{ x_0^2 }{2(2K+1)^2}  } \Big)
    \;.
\end{align*}
$\Delta_\#^-$ can be controlled by a similar argument:
\begin{align*}
    \Delta_\#^-
    \;=&\;
    e^{ - \frac{\beta^2 \log m}{2}  }
    \Big( 
        e^{- \frac{ y_{-1,+}^2  }{2 \beta^2 \log m}}
        +
        \msum_{j=-K}^{-2}
        e^{
            - \frac{ y_{j,+}^2  }{2 \beta^2 \log m}
        }
    \Big)
    \\
    \;\leq&\;
    e^{ - \frac{\beta^2 \log m}{2}  }
    \Big( 
        e^{- \frac{ y_{-1,+}^2  }{2 \beta^2 \log m}}
        +
        \mint_{-K}^{-1}
        e^{
            - \frac{ y_{u,+}^2  }{2 \beta^2 \log m}
        }
        du
    \Big)
    \\
    \;=&\;
    e^{ - \frac{\beta^2 \log m}{2}  }
    \Big( 
        e^{- \frac{ y_{-1,+}^2  }{2 \beta^2 \log m}}
        +
        \mint_{-K}^{-1}
        e^{
            - \frac{ (2u+1)^2  x_0^2 \log m }{2 (2K+1)^2 }
        }
        du
    \Big)
    \\
    \;=&\;
    e^{ - \frac{\beta^2 \log m}{2}  }
    \Big( 
        e^{- \frac{ y_{-1,+}^2  }{2 \beta^2 \log m}}
        +
        \mfrac{2K+1}{2x_0 \sqrt{\log m}}
        \mint_{-y_{-1,+} / (\beta \sqrt{\log m})}^{-y_{-K,+} / (\beta \sqrt{\log m})}
        e^{
            - \frac{u^2}{2 }
        }
        du
    \Big)
    \\
    \;\leq&\;
    e^{ - \frac{\beta^2 \log m}{2}  }
    \Big( 
        e^{- \frac{ y_{-1,+}^2  }{2 \beta^2 \log m}}
        +
        \mfrac{2K+1}{2x_0 \sqrt{\log m}}
        \Phi\Big( \mfrac{y_{-1,+}}{\beta \sqrt{\log m}} \Big)
    \Big)
    \\
    \;=&\;
    e^{ - \frac{\beta^2 \log m}{2}  }
    \Big( 
        e^{- \frac{ x_0^2 \log m  }{2 (2K+1)^2 }}
        +
        \mfrac{2K+1}{2x_0 \sqrt{\log m}}
        \Phi\Big( -  \mfrac{x_0 \sqrt{\log m}}{2K+1} \Big)
    \Big)
    \\
    \;=&\;
    O\Big(      
        e^{ - \frac{\beta^2 \log m}{2} - \frac{ x_0^2 \log m  }{2 (2K+1)^2 }  }
    \Big)
    \;=\;
    O\Big(      
        m^{ - \frac{\beta^2}{2} - \frac{ x_0^2  }{2 (2K+1)^2 }  }
    \Big)
    \;.
\end{align*}
Therefore
\begin{align*}
    \Delta_\# 
    \;=&\; 
     O\Big( 
        e^{o(\sqrt{\log m}) }
        \Big(
            \Delta_\#^-
            +
            O\big( e^{ - \frac{\beta^2 \log m}{2} } \big)
            +
            \Delta_\#^+
        \Big)
    \Big)
    \\
    \;=&\; 
     O\Big( 
        e^{o(\sqrt{\log m}) }
        \Big(
            m^{ - \frac{\beta^2}{2} }
            +
            m^{- \frac{ x_0^2  }{2 (2K+1)^2 }  }
        \Big)
    \Big)
    \;=\; o(1)\;.
\end{align*}
In the last line, we have used that $\beta = \Omega(1)$ and that since $K = \Theta((\log m)^{1/4})$,
\begin{align*}
    \mfrac{ x_0^2 \log m }{2 (2K+1)^2 }  \;=\; \Theta( \sqrt{\log m} )  \;.
\end{align*}
In this case, recalling the definition of $\Delta_\#$ and the computation \eqref{eq:calc:I:sharp:main}, we obtain
\begin{align*}
    I_-^\# 
    \;=&\; 
     (1+o(1)) \, 
     P(\cB_{j^\#})  \, e^{y_{j^\#,-}}
    \\
    \;=&\;
    (1+o(1))
    \,
    \exp\Big( - \mfrac{\delta_\#^2}{2}   
    +
    \mfrac{\beta^2 \log m}{2} 
     + o(\sqrt{\log m}) \Big)
    \;.
\end{align*}
Using additionally that $I_-^* = I_-^\# \,e^{ \frac{2 x_0 \beta \log m}{2K+1}  }$ from
\eqref{eq:relate:I:star:I:sharp}, we obtain 
\begin{align*}
    I_-
    \;\geq&\;
    (1- \epsilon)
    (1+o(1))
    \,
    \exp\Big( - \mfrac{\delta_\#^2}{2}   
    +
    \mfrac{\beta^2 \log m}{2} 
     + o(\sqrt{\log m}) \Big)
    \;,
    \\
    I_- 
    \;\leq&\;
    (1 + \epsilon)
    (1+o(1))
    \,
    \,
    \exp\Big( - \mfrac{\delta_\#^2}{2}   
    +
    \mfrac{\beta^2 \log m}{2} 
    +
    \mfrac{2 x_0 \beta \log m}{2K+1}
     + o(\sqrt{\log m}) \Big)
    \;.
\end{align*} 
Taking logarithm in both inequalities and rescaling both sides by $\frac{2}{\beta^2 \log m}$, we obtain
\begin{align*}
    \Big| \mfrac{2 \log I_-}{\beta^2 \log m}   - 1\Big|
    \;=&\;
    O\Big(
    \mfrac{\max\{|\log(1-\epsilon)|, |\log(1+\epsilon)|\} }{\beta^2 \log m}
    +
    \mfrac{\log(1 + o(1))}{\beta^2 \log m}
    +
    \mfrac{\delta_\#^2}{\beta^2 \log m}
    +
    \mfrac{1}{\beta K}
    \Big)
    \\
    &\;
    +
    o\Big( \mfrac{1}{\beta^2 \sqrt{\log m}} \Big)
    \\
    \;\overset{(a)}{=}&\;
    O\Big( 
        \mfrac{\epsilon}{(1-\epsilon) \log m}
        +
        \mfrac{o(1)}{\log m}
        +
        \mfrac{1}{ \sqrt{\log m}}
        +
        \mfrac{1}{(\log m)^{1/4}}
        +
        o\Big( \mfrac{1}{\sqrt{\log m}} \Big)
    \Big)
    \\
    \;=&\;
    O\Big( 
        \mfrac{\epsilon}{(1-\epsilon) \log m}
        +
        \mfrac{1}{(\log m)^{1/4}}
    \Big)
    \;.
\end{align*}
In $(a)$, we have recalled that  $\delta_\# = O((\log m)^{1/4})$ by \eqref{eq:delta:sharp:bound:one}, $K =\Theta((\log m)^{1/4})$ and $\beta = \Omega(1)$, and noted that for $\epsilon \in (0,1)$, $0 \leq \log(1+\epsilon) \leq \epsilon$ and $0 \geq\log(1-\epsilon) \geq  - \frac{\epsilon}{1-\epsilon}$.
\end{proof}

\vspace{1em}

The second lemma focuses on $\beta > x_0$.

\begin{lemma} \label{lem:I:minus:beta:larger} Assume $x_0 = \Theta(1)$ and $K=\Theta((\log m)^{1/4})$. For $\beta > x_0$, we have that conditioning on $A_\epsilon$ with $\epsilon \in (0,1)$,
\begin{align*}
    \Big| 
        \mfrac{ \log I_-}{x_0 \beta \log m - \frac{x_0^2}{2} \log m  } 
        \,-\, 
        1
    \Big|
    \;=&\;
    O\Big( 
        \mfrac{\epsilon}{(1-\epsilon)\log m}
        +
        \mfrac{1}{(\log m)^{1/4}}
    \Big)
    \;.
\end{align*}
\end{lemma}

\begin{proof} In the case $\beta > x_0$, $j^\# = K$ and $\delta_\#$ no longer satisfy the bound in \eqref{eq:delta:sharp:bound:one}. We sum the calculations of $P(\cB_j) e^{y_{j,-}} / P(\cB_{j^\#})  \, e^{y_{j^\#,-}}$  in \Cref{lem:compute:P:e}, express $\delta_\#$ explicitly and rearrange the terms as
\begin{align*}
    \Delta_\# 
    \;=&\;
    \Theta\bigg(
    \, 
    e^{o(\sqrt{\log m})}
    \,
    \bigg( 
        \underbrace{
            \sum_{j=-K}^{-1}
             \exp\Big( 
                - 
                \mfrac{y_{j,+}^2 - y_{j^\#, -}^2}{2 \beta^2 \log m} 
                +
                (y_{j,-} -  y_{j^\#, -}) 
            \Big) 
        }_{\tilde \Delta^-_\#}
        \,+\,
        \exp\Big(  
                \mfrac{y_{j^\#, -}^2}{2 \beta^2 \log m} 
                -
                y_{j^\#, -}
        \Big)
    \\
    &\hspace{7em}
        +
        \underbrace{
            \sum_{j=1}^{K-1}
            \exp\Big( 
                - 
                \mfrac{y_{j,-}^2 - y_{j^\#, -}^2}{2 \beta^2 \log m} 
                +
                (y_{j,-} -  y_{j^\#, -}) 
            \Big) 
        }_{\tilde \Delta^+_\#}
    \bigg)
    \,
    \bigg)
    \;.
\end{align*}
Recall that $y_{j^\#, -} = \frac{2K-1}{2K+1} x_0 \beta \log m$ and that, since $\beta > x_0$, 
\begin{align*}
    \beta - \mfrac{2K-1}{2K+1} \mfrac{x_0}{2}  
    \;>\; 
    \beta - \mfrac{x_0}{2} \;>\; \mfrac{x_0}{2}\;.
\end{align*}
We can control the second term above as
\begin{align*}
    \exp\Big(  
                \mfrac{y_{j^\#, -}^2}{2 \beta^2 \log m} 
                -
                y_{j^\#, -}
        \Big)
    \;=&\;
    \exp\Big(  
            \Big(
                \mfrac{y_{j^\#, -}}{2 \beta \log m} 
                -
                \beta
            \Big)
            \, 
            \mfrac{y_{j^\#, -}}{\beta}
    \Big)
    \\
    \;=&\;
    \exp\Big(  
            -
            \Big(
                \beta 
                -
                \mfrac{2K-1}{2K+1}
                \mfrac{x_0}{2} 
            \Big)
            \, 
            \mfrac{2K-1}{2K+1} x_0 \log m
    \Big)
    \\
    \;\leq&\;
    \exp\Big(  
            -
            \mfrac{x_0^2}{2} \mfrac{2K-1}{2K+1} \log m
    \Big)
    \;=\;
    m^{ - \frac{x_0^2}{2} \frac{2K-1}{2K+1} }
    \;.
\end{align*}
Meanwhile by noting that $(y_{j,-} - \beta^2 \log m)^2$ is decreasing in $j$,
\begin{align*}
    \tilde \Delta^+_\#
    \;=&\;
    \exp\Big( \mfrac{(  y_{j^\#,-} - \beta^2 \log m)^2}{2 \beta^2 \log m} \Big)
    \,
    \msum_{j=1}^{K-1}
    \exp\Big(  - \mfrac{(  y_{j,-} - \beta^2 \log m)^2}{2 \beta^2 \log m} \Big)
    \\
    \;\leq&\;
    \exp\Big( \mfrac{(  y_{j^\#,-} - \beta^2 \log m)^2}{2 \beta^2 \log m} \Big)
    \\
    &\quad \,\times\,
    \Big( 
        \exp\Big(  - \mfrac{(  y_{K-1,-} - \beta^2 \log m)^2}{2 \beta^2 \log m} \Big)
        +
        \mint_1^{K-1}
        \exp\Big(  - \mfrac{(  y_{u,-} - \beta^2 \log m)^2}{2 \beta^2 \log m} \Big)
        du
    \Big)
    \\
    \;=&\;
    \exp\Big( \mfrac{(  y_{j^\#,-} - \beta^2 \log m)^2}{2 \beta^2 \log m} \Big)
    \,\times\,
    \Big( 
        \exp\Big(  - \mfrac{(  y_{K-1,-} - \beta^2 \log m)^2}{2 \beta^2 \log m} \Big)
    \\
    &\hspace{11em}
        +
        \mint_1^{K-1}
            \exp\Big(  - \mfrac{( \frac{2u-1}{2K+1} x_0 \sqrt{\log m}  - \beta \sqrt{\log m})^2}{2} \Big)
        du
    \Big)
    \\
    \;=&\;
    \exp\Big( \mfrac{(  y_{j^\#,-} - \beta^2 \log m)^2}{2 \beta^2 \log m} \Big)
    \,\times\,
    \Big( 
        \exp\Big(  - \mfrac{(  y_{K-1,-} - \beta^2 \log m)^2}{2 \beta^2 \log m} \Big)
    \\
    &\hspace{11em}
        +
        \mfrac{2K+1}{2x_0 \sqrt{\log m}}
        \mint_{y_{1,-} / (\beta \sqrt{\log m}) - \beta \sqrt{\log m}}^{y_{K-1,-} / (\beta \sqrt{\log m}) - \beta \sqrt{\log m}}
            e^{- \frac{u^2}{2}}
        du
    \Big)
    \\
    \;\leq&\;
    \exp\Big( \mfrac{(  y_{j^\#,-} - \beta^2 \log m)^2}{2 \beta^2 \log m} \Big)
    \,\times\,
    \Big( 
        \exp\Big(  - \mfrac{(  y_{K-1,-} - \beta^2 \log m)^2}{2 \beta^2 \log m} \Big)
    \\
    &\hspace{13em}
        +
        \mfrac{2K+1}{2x_0 \sqrt{\log m}}
        \Phi\Big( - \mfrac{\beta^2 \log m - y_{K-1,-}}{\beta \sqrt{\log m}} \Big)
    \Big)
    \\
    \;=&\;
    \exp\Big( 
        \mfrac{(  y_{j^\#,-} - \beta^2 \log m)^2}{2 \beta^2 \log m} 
        - 
        \mfrac{(  y_{K-1,-} - \beta^2 \log m)^2}{2 \beta^2 \log m}
    \Big)
    \\
    &\qquad 
    \,\times\,
    \Big( 
        1
        +
        \mfrac{2K+1}{2x_0 \sqrt{\log m}}
        \, \mfrac{\beta \sqrt{\log m}}{\beta^2 \log m - y_{K-1,-}}
    \Big)
    \\
    \;\overset{(a)}{=}&\;
    O\Big( m^{ - \frac{2x_0^2}{(2K+1)^2} } \Big)
    \;.
\end{align*}
In $(a)$, we used $x_0 = \Theta(1)$, $K=\Theta((\log m)^{1/4})$, $y_{K-1,-} =  \frac{2K-3}{2K+1} x_0 \beta \log m$ and $\beta > x_0$ to compute 
\begin{align*}
    \mfrac{2K+1}{2x_0 \sqrt{\log m}}
    \, \mfrac{\beta \sqrt{\log m}}{\beta^2 \log m - y_{K-1,-}}
    \;=&\;
    \Theta\Big( \mfrac{1}{(\log m)^{1/4}} \mfrac{1}{\beta \sqrt{\log m} \,-\, \frac{2K-3}{2K+1} x_0 \sqrt{\log m} } \Big)
    \\
    \;=&\;
    \Theta\Big( \mfrac{1}{(\log m)^{1/4}} \mfrac{2K+1}{4 x_0 \sqrt{\log m}} \Big)
    \;=\;
    \Theta\Big( \mfrac{1}{\sqrt{\log m}} \Big)
    \;=\;
    o(1)
    \;,
\end{align*}
and also compute 
\begin{align*}
    &\;\mfrac{(  y_{j^\#,-} - \beta^2 \log m)^2}{2 \beta^2 \log m} 
    - 
    \mfrac{(  y_{K-1,-} - \beta^2 \log m)^2}{2 \beta^2 \log m}
    \\
    &\;=\;
    \mfrac{y_{K,-}^2 -  y_{K-1,-}^2}{2 \beta^2 \log m} 
    - 
    (y_{K,-} - y_{K-1,-})
    \\
    &\;=\;
    - 
    (y_{K,-} - y_{K-1,-}) \, \mfrac{ 2\beta^2 \log m - y_{K,-} - y_{K-1,-}}{2 \beta^2 \log m}  
    \\
    &\;=\;
    - 
    \mfrac{2 x_0 \beta \log m }{2K+1}
    \, \mfrac{ 2\beta^2 \log m - \frac{4K}{2K+1} x_0 \beta \log m }{2\beta^2 \log m} 
    \\ 
    &\;=\;
    - 
    \mfrac{2 x_0 \log m  }{2K+1}
    \,\Big( \beta - \mfrac{2K}{2K+1} x_0 \Big)
    \\
    &\;\leq\;
    - 
    \mfrac{2 x_0 \log m  }{2K+1}
    \,\Big( x_0 - \mfrac{2K}{2K+1} x_0 \Big)
    \;=\;
    - 
    \mfrac{2 x_0^2 \log m  }{(2K+1)^2}
    \;.
\end{align*}
Meanwhile, by using $y_{j,-} < y_{j,+}$ and applying a similar argument with the observation that $(  y_{j,+} - \beta^2 \log m)^2$ is decreasing in $j$, we obtain
\begin{align*}
    \tilde \Delta^-_\# 
    \;\leq&\;
    \msum_{j=-K}^{-1}
    \exp\Big( 
        - 
        \mfrac{y_{j,+}^2 - y_{j^\#, -}^2}{2 \beta^2 \log m} 
        +
        (y_{j,+} -  y_{j^\#, -}) 
    \Big) 
    \\
    \;=&\;
    \exp\Big( \mfrac{(  y_{j^\#,-} - \beta^2 \log m)^2}{2 \beta^2 \log m} \Big)
    \,
    \msum_{j=-K}^{-1}
    \exp\Big(  - \mfrac{(  y_{j,+} - \beta^2 \log m)^2}{2 \beta^2 \log m} \Big)
    \\
    \;\leq&\;
    \exp\Big( 
        \mfrac{(  y_{j^\#,-} - \beta^2 \log m)^2}{2 \beta^2 \log m} 
        - 
        \mfrac{(  y_{-1,+} - \beta^2 \log m)^2}{2 \beta^2 \log m}
    \Big)
    \\
    &\qquad 
    \,\times\,
    \Big( 
        1
        +
        \mfrac{2K+1}{2x_0 \sqrt{\log m}}
        \, \mfrac{\beta \sqrt{\log m}}{\beta^2 \log m - y_{-1,+}}
    \Big)
    \\
    \;\overset{(b)}{=}&\;
    O\Big( m^{ - \frac{2 K(K+2)x_0^2}{(2K+1)^2} } \Big)
    \;.
\end{align*}
In $(b)$, we have used $y_{-1,+} < 0$ to compute that
\begin{align*}
    \mfrac{2K+1}{2x_0 \sqrt{\log m}}
    \, \mfrac{\beta \sqrt{\log m}}{\beta^2 \log m - y_{-1,+}}
    \;=&\;
    O\Big( \mfrac{1}{(\log m)^{1/4}} \mfrac{1}{\beta \sqrt{\log m} } \Big)
    \;=\;
    o(1)
\end{align*}
and that 
\begin{align*}
    &\;\mfrac{(  y_{j^\#,-} - \beta^2 \log m)^2}{2 \beta^2 \log m} 
    - 
    \mfrac{(  y_{-1,+} - \beta^2 \log m)^2}{2 \beta^2 \log m}
    \\
    &\;=\;
    \mfrac{y_{K,-}^2 -  y_{-1,+}^2}{2 \beta^2 \log m} 
    - 
    (y_{K,-} - y_{-1,+})
    \\
    &\;=\;
    - 
    (y_{K,-} - y_{-1,+}) \, \mfrac{ 2\beta^2 \log m - y_{K,-} - y_{-1,+}}{2 \beta^2 \log m}  
    \\
    &\;=\;
    - 
    \mfrac{2K x_0 \beta \log m }{2K+1}
    \, \mfrac{ 2\beta^2 \log m - \frac{2K -2 }{2K+1} x_0 \beta \log m }{2\beta^2 \log m} 
    \\ 
    &\;=\;
    - 
    \mfrac{K x_0 \log m }{2K+1}
    \, 
    \Big(
     2\beta - \mfrac{2K -2 }{2K+1} x_0 
    \Big) 
    \\
    &\;\leq\;
    - 
    \mfrac{K x_0^2 \log m }{2K+1}
    \, 
    \mfrac{2K+4 }{2K+1} 
    \;\leq\;
    - 
    \mfrac{2 K(K+2) x_0^2 \log m  }{(2K+1)^2}
    \;.
\end{align*}
Combining the bounds, we obtain
\begin{align*}
    \Delta_\# \;=\; 
    O\Big( 
        e^{o(\sqrt{\log m})} 
        \,
        \Big(
            m^{ - \frac{2x_0^2}{(2K+1)^2} } 
            +
            m^{ - \frac{x_0^2}{2} \frac{2K-1}{2K+1} }
            +
            m^{ - \frac{2 K(K+2)x_0^2}{(2K+1)^2} }
        \Big)
    \Big)
    \;=\;
    o(1)\;,
\end{align*}
where we have noted that $\frac{\log m}{K^2} = \Theta(\sqrt{\log m})$. Again recalling the definition of $\Delta_\#$ and applying the computation \eqref{eq:calc:I:sharp:main}, we obtain
\begin{align*}
    I_-^\# 
    \;=&\; 
    (1+o(1))
    \,
    \exp\Big( - \mfrac{y_{K,-}^2}{2 \beta^2 \log m} + y_{K,-} + o(\sqrt{\log m}) \Big)
    \\
    \;=&\;
    (1+o(1))
    \,
    \exp\Big( - \mfrac{(2K-1)^2 x_0^2 \beta^2 (\log m)^2}{2 (2K+1)^2 \beta^2 \log m} 
    + 
    \mfrac{2K-1}{2K+1} x_0 \beta \log m
    + 
    o(\sqrt{\log m}) \Big)
    \\
    \;=&\;
    (1+o(1))
    \,
    \exp\Big( - \mfrac{(2K-1)^2 x_0^2 \log m}{2 (2K+1)^2 } 
    + 
    \mfrac{2K-1}{2K+1} x_0 \beta \log m
    + 
    o(\sqrt{\log m}) \Big)
    \;.
\end{align*}
Using additionally that $I_-^* = I_-^\# \,e^{ \frac{2 x_0 \beta \log m}{2K+1}  }$ from \eqref{eq:relate:I:star:I:sharp}, we obtain 
\begin{align*}
    I_-
    \;\geq&\;
    (1- \epsilon)
    (1+o(1))
    \,
    \exp\Big( - \mfrac{(2K-1)^2 x_0^2 \log m}{2 (2K+1)^2 } 
    + 
    \mfrac{2K-1}{2K+1} x_0 \beta \log m
    + 
    o(\sqrt{\log m}) \Big)
    \;,
    \\
    I_- 
    \;\leq&\;
    (1 + \epsilon)
    (1+o(1))
    \,
    \exp\Big( - \mfrac{(2K-1)^2 x_0^2 \log m}{2 (2K+1)^2} 
    + 
    x_0 \beta \log m
    + 
    o(\sqrt{\log m}) \Big)
    \;.
\end{align*} 
We again take logarithm in both inequalities, but now use the rescaling factor $(x_0 \beta \log m - \frac{x_0^2}{2} \log m )^{-1}$, which satisfies
\begin{align*}
    \mfrac{2}{x_0 \beta \log m } 
    \;<\;
    \mfrac{1}{x_0 \beta \log m - \frac{x_0^2}{2} \log m  } 
    \;<\; 
    \mfrac{2}{x_0^2 \log m }
\end{align*}
since $\beta > x_0$ and therefore 
\begin{align*}
    \mfrac{2}{x_0 \beta \log m }  \;=\; \Theta( (\log m)^{-1})\;.
\end{align*}
This gives 
\begin{align*}
    \Big| \;&\, 
        \mfrac{ \log I_-}{x_0 \beta \log m - \frac{x_0^2}{2} \log m  } 
        \,-\, 
        1
    \Big|
    \\
    \;=&\;
    O\Big(
    \mfrac{\max\{|\log(1-\epsilon)|, |\log(1+\epsilon)|\}}{\log m}
    +
    \mfrac{\log(1 + o(1))}{\log m}
    +
    o\Big( \mfrac{1}{\log m} \Big)
    \\
    &\qquad 
    +
    \mfrac{x_0 \beta \log m }{x_0 \beta \log m - \frac{x_0^2}{2} \log m }
    \,\times\,
    \Big| \mfrac{2K-1}{2K+1} - 1 \Big|
    +
    \mfrac{\frac{x_0^2 \log m}{2} }{x_0 \beta \log m - \frac{x_0^2}{2} \log m }
    \Big| \mfrac{(2K-1)^2}{(2K+1)^2}  - 1  \Big|
    \Big)
    \\
    \;=&\;
    O\Big(
    \mfrac{\epsilon}{(1-\epsilon) \log m}
    +
    o\Big( \mfrac{1}{\log m} \Big)
    +
    \mfrac{2 \beta}{2\beta - x_0 } \, \mfrac{1}{K}
    +
    \mfrac{x_0}{2\beta - x_0 } \, \mfrac{1}{K}
    \Big)
    \\
    \;=&\;
    O\Big(
    \mfrac{\epsilon}{(1-\epsilon) \log m}
    +
    o\Big( \mfrac{1}{\log m} \Big)
    +
    \mfrac{2 \beta}{\beta } \, \mfrac{1}{K}
    +
    \mfrac{x_0}{\beta } \, \mfrac{1}{K}
    \Big)
    \\
    \;=&\;
    O\Big( 
        \mfrac{\epsilon}{(1-\epsilon)\log m}
        +
        \mfrac{1}{(\log m)^{1/4}}
    \Big)
    \;.
\end{align*}
In the last two lines, we have used that $\beta > x_0$ and $K = \Theta((\log m)^{1/4})$. We have also used $\epsilon \in (0,1)$ such that $\log(1-\epsilon)$ is well-defined and for simplifying the terms $\log(1-\epsilon)$ and $\log(1+\epsilon)$.
\end{proof}

\vspace{1em}

We have all the ingredients to prove \Cref{prop:moderate:cold}.

\begin{proof}[Proof of \Cref{prop:moderate:cold}] Take $x_0 = \Theta(1)$ and $K= \Theta((\log m)^{1/4})$, which satisfy the conditions of \Cref{lem:ignore:I:plus,lem:conc:I:minus,lem:compute:P:e,lem:I:minus:beta:smaller,lem:I:minus:beta:larger}. For $\epsilon > 0$, \Cref{lem:ignore:I:plus,lem:conc:I:minus} provide the existence of an event $\tilde A_\epsilon \coloneqq A_+ \cap A_\epsilon$ such that 
\begin{align*}
    1 - \P(\tilde A_\epsilon) 
    \;=\; 
    O\Big( 
        m^{1 - \frac{x_0^2}{2} +o\big( \frac{1}{\sqrt{\log m}} \big)  }
        +    
        \epsilon^{-2} \, m^{ - 1 +  \frac{x_0^2}{2} \frac{(2K-1)^2}{(2K+1)^2}  + o(\frac{1}{\sqrt{\log m}}) }  
    \Big)
    \;.
\end{align*}
Conditioning on $\tilde A_\epsilon$, $I_+=0$, and by \Cref{lem:I:minus:beta:smaller,lem:I:minus:beta:larger},
\begin{align*}
    \Big| \mfrac{2 \log I_-}{\beta^2 \log m}   - 1\Big|
    \;=&\;
    O\Big( 
        \mfrac{\epsilon}{(1-\epsilon)  \log m}
        +
        \mfrac{1}{(\log m)^{1/4}}
    \Big)
    \qquad 
    \text{ for } \beta \in (0,x_0]\;,
    \\
    \Big| 
        \mfrac{ \log I_-}{x_0 \beta \log m - \frac{x_0^2}{2} \log m  } 
        \,-\, 
        1
    \Big|
    \;=&\;
    O\Big( 
        \mfrac{\epsilon}{(1-\epsilon)\log m}
        +
        \mfrac{1}{(\log m)^{1/4}}
    \Big)
    \qquad 
    \text{ for } \beta > x_0 
    \;.
\end{align*}
To combine the bounds, we choose 
\begin{align*}
    x_0 
    \;=&\;  
    \sqrt{2} 
    \, \mfrac{2K}{2K-1}
    \;.
\end{align*}
Noting that $K = \Theta((\log m)^{1/4})$, we get that  there exist some universal constants $C', c> 0$ such that
\begin{align*}
    1 - \P(\tilde A_\epsilon) 
    \;=&\; 
    O\Big( 
        m^{- \frac{1}{2K-1} +o\big( \frac{1}{\sqrt{\log m}} \big)  }
        +    
        \epsilon^{-2} \, m^{ - 1 +  \frac{(2K)^2}{(2K+1)^2}  + o(\frac{1}{\sqrt{\log m}}) }  
    \Big)
    \\
    \;\leq&\;\
    C'
    \big(
        e^{ - \frac{\log m}{2K-1} +o( \sqrt{\log m} )  }
        +    
        \epsilon^{-2} \, e^{ -  \frac{4K \log m}{(2K+1)^2}  + o(\sqrt{\log m}) }  
    \big)
    \\
    \;\leq&\;
    (1+ \epsilon^{-2}) e^{ - c  \, (\log m)^{3/4} }\;,
\end{align*}
where we have noted that the constant coefficient in front of the exponential can be removed by choosing a sufficiently small $c > 0$ since the LHS is bounded from above by $1$. Meanwhile under this choice of $x_0$, conditioning on $\tilde A_\epsilon$,  we can control $\log I_-$ in three cases:
\begin{proplist}
    \item $\beta \in (0, \sqrt{2}]$. In this case, $\beta \in (0, x_0)$, so 
    \begin{align*}
        \Big| \mfrac{\log I_-  - \frac{\beta^2}{2} \log m }{\beta \log m}   \Big|
        \;=&\;
        O\Big( 
            \mfrac{\epsilon}{(1-\epsilon) \log m}
            +
            \mfrac{1}{ (\log m)^{1/4}}
        \Big)
        \;;
    \end{align*}
    \item $\beta \in (\sqrt{2}, x_0]$. In this case, we apply the triangle inequality to obtain 
    \begin{align*}
        \Big| 
            \mfrac{\log I_- - (\sqrt{2} \beta \log m - \log m )
            }{\beta \log m}   
        \Big|
        \;\leq&\;
        \Big| 
            \mfrac{\log I_- - \frac{\beta^2}{2} \log m
            }{\beta \log m}   
        \Big|
        +
        \Big| 
            \mfrac{\beta}{2}
            -
            \Big(\sqrt{2} - \mfrac{1}{\beta}\Big)
        \Big|
        \\
        \;=&\;
        O\Big( 
            \mfrac{\epsilon}{(1-\epsilon)  \log m}
            +
            \mfrac{1}{(\log m)^{1/4}}
            +
            \mfrac{(\beta - \sqrt{2})^2}{2\beta} 
        \Big)
        \\
        \;=&\;
        O\Big( 
            \mfrac{\epsilon}{(1-\epsilon)  \log m}
            +
            \mfrac{1}{(\log m)^{1/4}}
            +
            \mfrac{(x_0 - \sqrt{2})^2}{2\beta} 
        \Big) 
        \\
        \;=&\;
        O\Big( 
            \mfrac{\epsilon}{(1-\epsilon)  \log m}
            +
            \mfrac{1}{(\log m)^{1/4}}
            +
            \mfrac{1}{(2K-1)^2} 
        \Big)
        \\
        \;=&\;
        O\Big( 
            \mfrac{\epsilon}{(1-\epsilon) \log m}
            +
            \mfrac{1}{(\log m)^{1/4}}
        \Big)
        \;,
    \end{align*}
    where we used $K = \Theta((\log m)^{1/4})$ in the last line;
    \item $\beta > x_0$. In this case, we apply the triangle inequality again to obtain
    \begin{align*}
        &\;
        \Big| 
            \mfrac{\log I_- - (\sqrt{2} \beta \log m - \log m )
            }{\beta \log m}   
        \Big|
        \\
        &\;\leq\;
        \Big| 
            \mfrac{\log I_- - (x_0 \beta \log m - \frac{x_0^2}{2} \log m )
            }{\beta \log m}   
        \Big|
        +
        \Big| 
            \mfrac{(x_0 \beta \log m - \frac{x_0^2}{2} \log m ) - (\sqrt{2} \beta \log m - \log m )
            }{\beta \log m}   
        \Big|
        \\
        &\;=\; 
        O\Big( 
            \mfrac{x_0 \beta \log m - \frac{x_0^2}{2} \log m}{\beta \log m}
            \Big(
                \mfrac{\epsilon}{(1-\epsilon)\log m}
                +
                \mfrac{1}{(\log m)^{1/4}}
            \Big)
            +
            |x_0 - \sqrt{2}|
            +
            \mfrac{|x_0^2 - 2|}{\beta}
        \Big)
        \\
        &\;=\; 
        O\Big( 
            \mfrac{\sqrt{2} \beta - 1}{\beta }
            \Big(
                \mfrac{\epsilon}{(1-\epsilon)\log m}
                +
                \mfrac{1}{(\log m)^{1/4}}
            \Big)
            +
            |x_0 - \sqrt{2}|
            +
            \mfrac{|x_0^2 - 2|}{\beta}
        \Big)
        \\
        &\;=\; 
        O\Big( 
            \mfrac{\epsilon}{(1-\epsilon)\log m}
            +
            \mfrac{1}{(\log m)^{1/4}}
        \Big)\;,
    \end{align*}
    where we used that $x_0 - \sqrt{2} = O\big( \frac{1}{(2K-1)} \big) = O( (\log m)^{-1/4} )$ as well as $\beta = \Omega(1)$ in the last line.
\end{proplist} 
Finally, recall that $e^{\cE - 2(N-1) \mu} = I_- + I_+$, that $2(N-1) \mu = - \frac{N-1}{n} = - \frac{\beta^2 \log m}{2}$ and that 
\begin{align*}
    \bar \cE
    \;=&\;
    \begin{cases}
            0
            &
            \text{ if }
            \beta \leq \sqrt{2}
            \;,
            \\
            - \frac{(\beta - \sqrt{2})^2}{2} \, \log m
            & 
            \text{ if } 
            \beta > \sqrt{2}
            \;.
    \end{cases}
\end{align*}
Conditioning on $\tilde A_\epsilon$, we have $I_+ = 0$ and therefore in the case $\beta > \sqrt{2}$,
\begin{align*}
    \mfrac{|\cE - \bar \cE|}{\beta \log m} 
    \;=&\;
    \Big| \mfrac{\cE - 2(N-1)\mu - \frac{\beta^2 \log m}{2} + \frac{(\beta - \sqrt{2})^2}{2} \, \log m}{\beta \log m}   \Big|
    \\
    \;=&\;
    \Big| \mfrac{\log I_- - (\sqrt{2} \beta \log m - \log m)}{\beta \log m}   \Big|
    \;=\;
    O\Big( 
            \mfrac{\epsilon}{(1-\epsilon)\log m}
            +
            \mfrac{1}{(\log m)^{1/4}}
        \Big)\;,
\end{align*}
and in the case $\beta \leq \sqrt{2}$ with $\beta = \Omega(1)$,
\begin{align*}
    \mfrac{|\cE - \bar \cE|}{\beta \log m} 
    \;=&\;
    \Big| \mfrac{\cE - 2(N-1)\mu - \frac{\beta^2 \log m}{2}}{\beta \log m}   \Big|
    \\
    \;=&\;
    \Big| \mfrac{\log I_- -  \frac{\beta^2 \log m}{2}}{\beta \log m}   \Big|
    \;=\;
    O\Big( 
            \mfrac{\epsilon}{(1-\epsilon)\log m}
            +
            \mfrac{1}{(\log m)^{1/4}}
        \Big)\;.
\end{align*}
In summary, we have shown that there exist some universal constants $C, c > 0$ such that for every $\epsilon > 0$,
\begin{align*}
    \P\Big( 
        \mfrac{|\cE - \bar \cE|}{\beta \log m} 
        \;>\;
        C \Big( \mfrac{\epsilon}{(1-\epsilon)\log m}
            +
            \mfrac{1}{(\log m)^{1/4}} \Big)
    \Big)
    \;\leq\; 
    (1+ \epsilon^{-2}) e^{ - c  \, (\log m)^{3/4} }\;.
\end{align*}
Dividing both sides of the inequality inside $\P(\argdot)$ by $\beta$ completes the proof.
\end{proof}

\section{Proofs for the main results} \label{sec:proof:main}

The main results can now be obtained by combining \Cref{cor:log:Hanson:Wright,lem:sup:to:pointwise} with the REM calculations in \Cref{prop:ultrahot,prop:moderate:cold}. Throughout this section, $c, C > 0$ are universal constants whose values change from line to line.

\vspace{.5em}

In all the proofs, we recall the observation \eqref{eq:REM:first:derivation} that, for any fixed $\theta \in \cS^{n-1}$,
\begin{align*}
    \log\Big( \mfrac{1}{m} \msum_{i \leq m}  \| X_{i,N-1:1} \theta  \|^2 \Big)
    \;\overset{d}{=}\;
    \log\Big( 
        \mfrac{1}{m} \msum_{i \leq m} 
    e^{2 \sum_{j=1}^{N-1} Y_{ij} }
    \Big)
    \;=\;
    \cE
    \;,
\end{align*}
where $(Y_{ij})_{1 \leq i \leq m, 1 \leq j \leq N-1}$ are i.i.d.~random variables each distributed as $\frac{1}{2} \log\big( \frac{1}{n} \chi^2_n \big)$, and $\cE$ is defined as in \Cref{sec:REM}. \Cref{prop:ultrahot,prop:moderate:cold} can therefore be restated for any fixed $\theta \in \cS^{n-1}$ as follows:
\begin{proplist}
    \item If $\beta = o(1)$, then for any $\epsilon \in (0,1)$, the following statement holds with probability $1- \epsilon^{-2} e^{ - c \log m }$:
    \begin{align*}
            \Big| \log\Big( \mfrac{1}{m} \msum_{i \leq m}  \| X_{i,N-1:1} \theta  \|^2 \Big) - \bar \cE \Big|
            \;\leq\;  
            C
            \Big(
                \mfrac{\epsilon}{1-\epsilon}    
                +
                \mfrac{\beta^2 \log m}{n}
            \Big)
        \;;
        \tagaligneq\label{eq:final:ultrahot}
    \end{align*}
    \item If $\beta = \Omega(1)$, $\log m = o(N^{1/3})$ and $N=o(n^3)$,  then for any $\epsilon \in (0,1)$, the following statement holds with probability $1 - (1+ \epsilon^{-2}) e^{ - c  \, (\log m)^{3/4} }$:
    \begin{align*}
            \mfrac{1}{\beta^2 \log m} \Big| \log\Big( \mfrac{1}{m} \msum_{i \leq m}  \| X_{i,N-1:1} \theta  \|^2 \Big)  - \bar \cE \Big|
            \;\leq\;
            C \Big( \mfrac{\epsilon}{(1-\epsilon) \beta \log m}
                +
                \mfrac{1}{\beta (\log m)^{1/4}}
            \Big)
            \;.
            \tagaligneq\label{eq:final:moderate:cold}
    \end{align*}
\end{proplist}
We also recall that $Z$ defined in \eqref{eq:limit} satisfies $Z = \frac{\bar \cE}{2}$.

\subsection{Proof of \Cref{thm:pointwise}} 
By \Cref{cor:log:Hanson:Wright}, we have that for every $\epsilon \in (0,1)$ and $\theta \in \cS^{n-1}$,
\begin{align*}
    \P\Big( 
        \Big|  
            \log \| X \theta \|
            -
            \mfrac{1}{2} \log  
            \Big(
                \mfrac{1}{m} \msum_{i \leq m} 
                \| X_{i,N-1:1} \theta \|^2
            \Big)
        \Big| 
        \;\geq\; 
        \mfrac{\epsilon}{2(1-\epsilon)}
    \Big)
    \;\leq\;
    2 e^{-c n \epsilon^2 }
    \;.
\end{align*}
Combining this with \eqref{eq:final:ultrahot} by a union bound and the triangle inequality, followed by renaming universal constants, we obtain that for $\beta = o(1)$ and any $\epsilon \in (0,1)$, with probability $1 - \epsilon^{-2} e^{ - c \log m } - 2 e^{-cn\epsilon^2}$, we have
\begin{align*}
    \big|  
        \log \| X \theta \|
        -
        Z
    \big| 
    \;\leq\; 
     C
    \Big(
        \mfrac{\epsilon}{1-\epsilon}
        +
        \mfrac{\beta^2 \log m}{n}
    \Big)
    \;.
\end{align*}
The same argument with \eqref{eq:final:moderate:cold} gives that, for $\beta = \Omega(1)$, $\log m = o(N^{1/3})$, $N=o(n^3)$ and any $\epsilon \in (0,1)$, with probability $1 - (1+ \epsilon^{-2}) e^{ - c  \, (\log m)^{3/4} } - 2e^{-cn\epsilon^2}$, we have
\begin{align*}
    \mfrac{1}{\beta^2 \log m}
    \big|  
        \log \| X \theta \|
        -
        Z
    \big| 
    \;\leq\;
    C \Big( \mfrac{\epsilon}{(1-\epsilon)  \beta \log m}
            +
                \mfrac{1}{  \beta (\log m)^{1/4}}
            \Big)
    \;.
\end{align*}\qed

\subsection{Proof of \Cref{thm:top:Lyapunov}}
Recall that by \Cref{lem:sup:to:pointwise}, there exists a universal constant $C' > 0$ such that, for any $\epsilon \in (0,1)$ and $\theta \in \cS^{n-1}$, we have
\begin{align*}
    \P\Big( 
        \big| 
            \log \| X \theta \|
            - 
            \log s_1(X)
        \big|  
        \,\geq\,
        \mfrac{1}{2} \log \Big( \mfrac{n}{\epsilon^2} \Big)
    \Big)
    \leq 
    C' \epsilon^{1/2}
    \;,
\end{align*}
where we have replaced $C^{1/2}$ in \Cref{lem:sup:to:pointwise} by $C'$ and recalled that 
\begin{align*}
    \log s_1(X) \;=\; \sup_{\theta' \in \cS^{n-1}} \log \| X \theta' \|\;.
\end{align*}
For any $\alpha > 0$, choose $\epsilon = n^{-\alpha}$. This implies 
\begin{align*}
    \P\Big( 
        \big| 
            \log \| X \theta \|
            - 
            \log s_1(X)
        \big|  
        \,\geq\,
        \mfrac{(1+2\alpha) \log n}{2} 
    \Big)
    \leq 
    C' n^{ - \alpha / 2 }\;.
\end{align*}
Combining this with \Cref{thm:pointwise} finishes the proof.  \qed

\vspace{3em}

\noindent
\textbf{Acknowledgements.} KHH gratefully acknowledges support from the UK Engineering and Physical Sciences Research Council (EPSRC) (Grant No.~EP/Y028783/1, Prob\_AI Hub). BH gratefully acknowledges support from a 2024 Sloan Fellowship in Mathematics, NSF CAREER grant DMS-2143754, NSF grant DMS-2133806, and DARPA AIQ grant (HR001124S0029). 

\bibliography{ref}

@article{hanin2021non,
  title={Non-asymptotic results for singular values of {G}aussian matrix products},
  author={Hanin, Boris and Paouris, Grigoris},
  journal={Geometric and Functional Analysis},
  volume={31},
  number={2},
  pages={268--324},
  year={2021},
  publisher={Springer}
}

@article{dorlas2001large,
  title={Large deviations and the random energy model},
  author={Dorlas, Teunis C and Wedagedera, Janak R},
  journal={International Journal of Modern Physics B},
  volume={15},
  number={01},
  pages={1--15},
  year={2001},
  publisher={World Scientific}
}

@article{bovier2002fluctuations,
  title={Fluctuations of the free energy in the {REM} and the $ p $-spin {SK} models},
  author={Bovier, Anton and Kurkova, Irina and L{\"o}we, Matthias},
  journal={The Annals of Probability},
  volume={30},
  number={2},
  pages={605--651},
  year={2002},
  publisher={Institute of Mathematical Statistics}
}

@article{derrida1981random,
  title={Random-energy model: An exactly solvable model of disordered systems},
  author={Derrida, Bernard},
  journal={Physical Review B},
  volume={24},
  number={5},
  pages={2613},
  year={1981},
  publisher={APS}
}

@article{hanin2025global,
  title={Global Universality of Singular Values in Products of Many Large Random Matrices},
  author={Hanin, Boris and Jiang, Tianze},
  journal={arXiv preprint arXiv:2503.07872},
  year={2025}
}

@article{hanin2020products,
  title={Products of many large random matrices and gradients in deep neural networks},
  author={Hanin, Boris and Nica, Mihai},
  journal={Communications in Mathematical Physics},
  volume={376},
  number={1},
  pages={287--322},
  year={2020},
  publisher={Springer}
}

@article{cramer1938nouveau,
  title={Sur un nouveau th{\'e}oreme-limite de la th{\'e}orie des probabilit{\'e}s},
  author={Cram{\'e}r, Harald},
  journal={Actualités Scientifiques et Industrielles},
  pages={5--23},
  volume={736},
  year={1938}
}

@book{petrov1975sums,
  title={Sums of independent random variables},
  author={Petrov, Valentin V},
  volume={82},
  year={1975},
  publisher={Springer Science \& Business Media}
}

@article{liu2023cramer,
  title={{C}ram{\'e}r-type moderate deviations under local dependence},
  author={Liu, Song-Hao and Zhang, Zhuo-Song},
  journal={The Annals of Applied Probability},
  volume={33},
  number={6A},
  pages={4747--4797},
  year={2023},
  publisher={Institute of Mathematical Statistics}
}

@book{vershynin2018high,
  title={High-dimensional probability: An introduction with applications in data science},
  author={Vershynin, Roman},
  volume={47},
  year={2018},
  publisher={Cambridge university press}
}

@article{liu2023lyapunov,
  title={{L}yapunov exponent, universality and phase transition for products of random matrices},
  author={Liu, Dang-Zheng and Wang, Dong and Wang, Yanhui},
  journal={Communications in Mathematical Physics},
  volume={399},
  number={3},
  pages={1811--1855},
  year={2023},
  publisher={Springer Nature BV}
}

@article{Ahn2022,
  author  = {Ahn, Andrew},
  title   = {Fluctuations of {$\beta$}-{J}acobi product processes},
  journal = {Probability Theory and Related Fields},
  volume  = {183},
  pages   = {57--123},
  year    = {2022},
}

@article{akemann2019integrable,
  title={From integrable to chaotic systems: Universal local statistics of {L}yapunov exponents},
  author={Akemann, Gernot and Burda, Zdzislaw and Kieburg, Mario},
  journal={Europhysics Letters},
  volume={126},
  number={4},
  pages={40001},
  year={2019},
  publisher={EDP Sciences, IOP Publishing and Societ{\`a} Italiana di Fisica}
}

@article{geman1980limit,
  title={A limit theorem for the norm of random matrices},
  author={Geman, Stuart},
  journal={The Annals of Probability},
  pages={252--261},
  year={1980},
  publisher={JSTOR}
}

@article{yin1988limit,
  title={On the limit of the largest eigenvalue of the large dimensional sample covariance matrix},
  author={Yin, Yong-Qua and Bai, Zhi-Dong and Krishnaiah, Pathak R},
  journal={Probability theory and related fields},
  volume={78},
  number={4},
  pages={509--521},
  year={1988},
  publisher={Springer}
}

@article{akemann2013products,
  title={Products of rectangular random matrices: singular values and progressive scattering},
  author={Akemann, Gernot and Ipsen, Jesper R and Kieburg, Mario},
  journal={Physical Review E—Statistical, Nonlinear, and Soft Matter Physics},
  volume={88},
  number={5},
  pages={052118},
  year={2013},
  publisher={APS}
}

@article{saada2024simple,
  title={A simple proof of almost sure convergence for the largest singular value of a product of {G}aussian matrices},
  author={Saada, Thiziri Nait and Naderi, Alireza},
  journal={arXiv preprint arXiv:2409.20180},
  year={2024}
}

@article{furstenberg1960products,
  title={Products of random matrices},
  author={Furstenberg, Harry and Kesten, Harry},
  journal={The Annals of Mathematical Statistics},
  volume={31},
  number={2},
  pages={457--469},
  year={1960},
  publisher={JSTOR}
}

@article{oseledets1968multiplicative,
  title={A multiplicative ergodic theorem. {C}haracteristic {L}japunov, exponents of dynamical systems},
  author={Oseledets, Valery Iustinovich},
  journal={Trudy Moskovskogo Matematicheskogo Obshchestva},
  volume={19},
  pages={179--210},
  year={1968},
  publisher={Moscow Mathematical Society}
}

@article{newman1986distribution,
  title={The distribution of {L}yapunov exponents: exact results for random matrices},
  author={Newman, Charles M},
  journal={Communications in mathematical physics},
  volume={103},
  number={1},
  pages={121--126},
  year={1986},
  publisher={Springer}
}

@article{cohen1984stability,
  title={The stability of large random matrices and their products},
  author={Cohen, Joel E and Newman, Charles M},
  journal={The Annals of Probability},
  pages={283--310},
  year={1984},
  publisher={JSTOR}
}

@article{isopi1992triangle,
  title={The triangle law for {L}yapunov exponents of large random matrices},
  author={Isopi, Marco and Newman, Charles M},
  journal={Communications in mathematical physics},
  volume={143},
  number={3},
  pages={591--598},
  year={1992},
  publisher={Springer}
}

@article{kargin2014largest,
  title={On the largest {L}yapunov exponent for products of {G}aussian matrices},
  author={Kargin, Vladislav},
  journal={Journal of Statistical Physics},
  volume={157},
  number={1},
  pages={70--83},
  year={2014},
  publisher={Springer}
}

@article{akemann2014universal,
  title={Universal distribution of {L}yapunov exponents for products of {G}inibre matrices},
  author={Akemann, Gernot and Burda, Zdzislaw and Kieburg, Mario},
  journal={Journal of Physics A: Mathematical and Theoretical},
  volume={47},
  number={39},
  pages={395202},
  year={2014},
  publisher={IOP Publishing}
}

@article{akemann2020universality,
  title={Universality of local spectral statistics of products of random matrices},
  author={Akemann, Gernot and Burda, Zdzislaw and Kieburg, Mario},
  journal={Physical Review E},
  volume={102},
  number={5},
  pages={052134},
  year={2020},
  publisher={APS}
}

@article{ahn2022extremal,
  title   = {Extremal singular values of random matrix products and {B}rownian motion on {GL}({N}, {$\mathbb{{C}}$})},
  author  = {Ahn, Andrew},
  journal = {Probability Theory and Related Fields},
  volume  = {187},
  number  = {3-4},
  pages   = {949--997},
  year    = {2023},
}

@article{haagerup2005new,
  title={A new application of random matrices: {$\operatorname{Ext}(C_{\mathrm{red}}^{\ast}(F_2))$} is not a group},
  author={Haagerup, Uffe and Thorbj{\o}rnsen, Steen},
  journal={The Annals of Mathematics},
  pages={711--775},
  year={2005},
  publisher={JSTOR}
}

@article{schultz2005non,
  title={Non-commutative polynomials of independent {G}aussian random matrices. {T}he real and symplectic cases.},
  author={Schultz, Hanne},
  journal={Probability theory and related fields},
  volume={131},
  number={2},
  pages={261--309},
  year={2005},
  publisher={Springer}
}

@incollection{van2025strong,
  author       = {van Handel, Ramon},
  title        = {The strong convergence phenomenon},
  booktitle    = {Current Developments in Mathematics, 2025},
  pages        = {177--261},
  publisher    = {International Press of Boston},
  address      = {Somerville, MA},
  year         = {2026},
}

@article{gotze2015asymptotic,
  title={Asymptotic spectra of matrix-valued functions of independent random matrices and free probability},
  author={G{\"o}tze, Friedrich and K{\"o}sters, Holger and Tikhomirov, Alexander},
  journal={Random Matrices: Theory and Applications},
  volume={4},
  number={02},
  pages={1550005},
  year={2015},
  publisher={World Scientific}
}

@article{kosters2015limiting,
  author  = {K{\"o}sters, Holger and Tikhomirov, Alexander},
  title   = {Limiting Spectral Distributions of Sums of Products of Non-Hermitian Random Matrices},
  journal = {Probability and Mathematical Statistics},
  volume  = {38},
  number  = {2},
  pages   = {359--384},
  year    = {2018},
  doi     = {10.19195/0208-4147.38.2.6}
}

@article{bordenave2011spectrum,
  author        = {Bordenave, Charles},
  title         = {On the spectrum of sum and product of non-{H}ermitian random matrices},
  journal       = {Electronic Communications in Probability},
  volume        = {16},
  pages         = {104--113},
  year          = {2011},
}

\end{document}